\documentclass[12pt]{amsart}
\usepackage{amsmath,amssymb,amsthm,graphics,amscd,mathrsfs}
\usepackage{amscd, amsfonts, amssymb, amsthm}
\usepackage{float, fullpage, verbatim, bbm}
\usepackage[colorlinks, breaklinks, linkcolor=blue]{hyperref}
\usepackage{breakurl}
\usepackage{array}
\usepackage{latexsym}
\usepackage{enumerate}

\usepackage{xcolor}
\usepackage{subcaption}
\usepackage{tikz-cd}
\usepackage{tikz}
\usetikzlibrary{arrows}
\usetikzlibrary{shapes}
\usepackage{tkz-graph}
\usetikzlibrary{decorations}
\usetikzlibrary{arrows,decorations.markings}
\usetikzlibrary{calc}
\usepackage{multirow}
\tikzstyle{v} = [circle, draw, inner sep=2pt, minimum size=3pt, fill=black]
\tikzstyle{l} = [rectangle, draw, rounded corners]

\usepackage{etoolbox}
\usepackage{caption}
\usepackage[parfill]{parskip}

\newcommand\CA{{\mathscr A}}
\newcommand\CB{{\mathscr B}}
\newcommand\CC{{\mathscr C}}

\newcommand\CS{{\mathcal S}}
\newcommand\CIF{{\mathcal {IF}}}

\newcommand\CIFAC{{\mathcal {IF\!AC}}}

\newcommand\CI{{\mathcal I}}

\newcommand\CM{{\mathcal M}}

\newcommand\R{{\varrho}}
\newcommand\RR{{\mathscr R}}

\newcommand\BBC{{\mathbb C}}
\newcommand\BBK{{\mathbb K}}
\newcommand\BBN{{\mathbb N}}
\newcommand\BBP{{\mathbb P}}
\newcommand\BBQ{{\mathbb Q}}
\newcommand\BBR{{\mathbb R}}
\newcommand\BBZ{{\mathbb Z}}

\newcommand\Aut{{\operatorname{Aut}}}
\newcommand\codim{\operatorname{codim}}

\newcommand\GL{\operatorname{GL}}

\newcommand\Poin{{\operatorname{Poin}}}

\newcommand\rank{\operatorname{rank}}
\newcommand\rk{\operatorname{rk}}

\newcommand\hgt{\operatorname{ht}}
\newcommand\Gen{\operatorname{Gen}}
\newcommand\lc{{\operatorname{lc}}}

\numberwithin{equation}{section}

\theoremstyle{plain}
\newtheorem{lemma}[equation]{Lemma}
\newtheorem{theorem}[equation]{Theorem}

\newtheorem{conjecture}[equation]{Conjecture}

\newtheorem{corollary}[equation]{Corollary}
\newtheorem{prop}[equation]{Proposition}
\theoremstyle{definition}
\newtheorem{defn}[equation]{Definition}
\newtheorem{remark}[equation]{Remark}

\newtheorem{example}[equation]{Example}


\begin{document}

\title[On connected subgraph arrangements]
{On connected subgraph arrangements}

\author[L. Giordani]{Lorenzo Giordani}
\address
{Fakult\"at f\"ur Mathematik,
	Ruhr-Universit\"at Bochum,
	D-44780 Bochum, Germany}
\email{lorenzo.giordani@rub.de}

\author[T. M\"oller]{Tilman M\"oller}
\address
{Fakult\"at f\"ur Mathematik,
	Ruhr-Universit\"at Bochum,
	D-44780 Bochum, Germany}
\email{tilman.moeller@rub.de}

\author[P.~M\"ucksch]{Paul M\"ucksch}
\address
{Institut für Algebra, Zahlentheorie und Diskrete Mathematik, 
	Fakultät für Mathematik und Physik, 
	Leibniz Universität Hannover, 
	Welfengarten 1, D-30167 Hannover, Germany}
\email{muecksch@math.uni-hannover.de}

\author[G. R\"ohrle]{Gerhard R\"ohrle}
\address
{Fakult\"at f\"ur Mathematik,
Ruhr-Universit\"at Bochum,
D-44780 Bochum, Germany}
\email{gerhard.roehrle@rub.de}


\allowdisplaybreaks

\begin{abstract}
In \cite{cuntzkuehne:subgraphs},
Cuntz and K\"uhne
introduced a particular class of hyperplane arrangements stemming from a given graph, so called
\emph{connected subgraph arrangements}.
In this note we strengthen some of the result from  \cite{cuntzkuehne:subgraphs}
and prove new ones for members of this class.
For instance, we show that aspherical members within this class stem from a rather restricted 
set of graphs. Specifically, if $\CA_G$ is an aspherical connected subgraph arrangement, then $\CA_G$ is free with the unique possible exception when the underlying graph $G$ is the complete graph on $4$ nodes.
\end{abstract}

\maketitle

\setcounter{tocdepth}{1}

\tableofcontents

\section{Introduction and Main Results}

In \cite{cuntzkuehne:subgraphs},
Cuntz and K\"uhne
introduced a special class of hyperplane arrangements stemming from a given graph, so called
\emph{connected subgraph arrangements}.
They include such prominent classes as the braid arrangements and resonance arrangements among many others.
In this note 
we strengthen some of the results proved in \cite{cuntzkuehne:subgraphs}
and prove some new theorems for members of this class.

We first recall the definition from \cite{cuntzkuehne:subgraphs}.
Fix $n \in \BBN$. Let $\BBK$ be a field and let $V = \BBK^n$. Let $x_1, \ldots, x_n$ be the dual basis  in $V^*$ of the standard $\BBK$-basis of $V$.
Let $G = (N,E)$ be an undirected graph with vertex set $N = [n]$ and edge set $E$. For $I \subseteq N$, let $G[I]$ be the induced subgraph of $G$ on the set of vertices $I$.
For $I \subseteq N$, define the hyperplane
\[H_I := \ker \sum_{i \in I}x_i.\]

\begin{defn}
	\label{def:AG} With the notation as above, the \emph{connected subgraph arrangement} $\CA_G(\BBK)$ in $V$ is defined as
	\[\CA_G(\BBK) := \{H_I \mid \varnothing \ne I \subseteq N \text{ if $G[I]$ is connected} \}.
	\]
	In case $\BBK = \BBQ$, we write $\CA_G := \CA_G(\BBQ)$.
\end{defn}

In \cite{cuntzkuehne:subgraphs},
Cuntz and K\"uhne classified all connected subgraph arrangements over $\BBQ$ which are free, factored, simplicial or supersolvable.
The purpose of this note is firstly to strengthen several of these results and secondly to prove new ones.

All properties of arrangements we study in this note are compatible with the product construction for arrangements. For a simple graph $G$, the connected subgraph arrangement $\CA_G$ is the product of the connected subgraph arrangements stemming from the connected components of $G$. Thus there is no harm in assuming that $G$ is connected throughout.

A property for arrangements is said to be \emph{combinatorial} if it only depends on the intersection lattice of the underlying arrangement.
Our first main theorem shows that the class of connected subgraph arrangements is
very special in the sense that essentially every property
we may formulate for members of this class is
combinatorial. This is formally captured by the notion of projective uniqueness, see Definition \ref{def:projunique}.

\begin{theorem}
	\label{thm:AGiscomb}
	Let $G$ be a connected graph. Then $\CA_G$ is projectively unique. 
\end{theorem}

Theorem \ref{thm:AGiscomb} implies that 
for the class of all rational arrangements whose underlying matroid admits a realization over $\BBQ$ as a connected subgraph arrangement 
freeness is combinatorial. In particular, Terao's conjecture over $\BBQ$ is valid within this class, cf.~\cite[Prop.~2.3]{ziegler:matroid}. Likewise, asphericity is combinatorial within this class. Whether both these properties are combinatorial in general are longstanding and wide open problems, see  \cite[Conj.~4.138]{orlikterao:arrangements} and \cite[Prob.~3.8]{falkrandell:homotopyII}.

Theorem \ref{thm:AGiscomb} is proved  in Section \ref{s:thm:AGiscomb}.  The results there show that the theorem holds for any fixed field $\BBK$.

In \cite[Thm.~1.6]{cuntzkuehne:subgraphs} Cuntz and K\"uhne classified all free members among the $\CA_G$:

\begin{theorem}
	\label{thm:freeAG}
	The connected subgraph arrangement $\CA_G$ is free if and only if $G$ is a path graph, a cycle graph, an almost-path graph, or a path-with-triangle graph (Definition \ref{definition: cuntz kühne graph families}).
\end{theorem}

We strengthen Theorem \ref{thm:freeAG} as follows, where we look at the stronger freeness property of \emph{accuracy}, see Definition \ref{def:accurate}.

\begin{theorem}
	\label{thm:accurateAG}
	The arrangement $\CA_G$ is free if and only if it is accurate.
\end{theorem}

Thanks to Theorem \ref{thm:AGiscomb}, the concept of accuracy is combinatorial within the class of connected subgraph arrangements. It is not clear whether this notion is combinatorial in general, see \cite[Rem.~1.3(ii)]{mueckschroehrletran:flagaccurate}.
Theorem \ref{thm:accurateAG}
is proved in \S \ref{s:accurateAG}.

While it follows from Theorem \ref{thm:AGiscomb} that freeness within the class of connected subgraph arrangements is a combinatorial property, it is viable to ask whether each of the families among the free connected subgraph arrangements belongs to a general class satisfying
a stronger combinatorial freeness property. For instance, if $G$ is a cycle graph or a path graph, then $\CA_G$ is inductively free \cite[Prop.~3.1, Cor.~8.11]{cuntzkuehne:subgraphs}.
For the notion of 
\emph{inductive} freeness, see Definition
\ref{def:indfree}. 
Moreover, thanks to  \cite[Thm.~4.1]{cuntzkuehne:subgraphs}, $\CA_G$ is MAT-free for $G = A_{n,k}$. 
For the notion of \emph{MAT-freeness}, see Definition
\ref{def:MATfree}. 
MAT-freeness is a combinatorial property for arrangements in characteristic zero, \cite[Lem.~18]{Cunzmuecksch:MATfree}.
Cuntz and M\"ucksch raised the question whether
MAT-freeness implies inductive freeness \cite[Ques.~3]{Cunzmuecksch:MATfree}.
This appears to be the case 
for the class of connected subgraph arrangements. Computational evidence up to rank $8$ suggests the following strengthening of 
Theorem \ref{thm:freeAG}.

\begin{conjecture}
	\label{conj:freeAG2}
	The 
	arrangement $\CA_G$ is free if and only if it is inductively free.
\end{conjecture}

It is not known whether the family 
$\CA_G$, where $G = \Delta_{n,k}$ is a path-with-triangle graph (Definition \ref{definition: cuntz kühne graph families})
satisfies a stronger combinatorial freeness property. We give further evidence for Conjecture \ref{conj:freeAG2} in \S \ref{s:conj:freeAG2}.

Note, an analogue of Theorem \ref{thm:freeAG} for free multiplicities on $\CA_G$ is studied in \cite{mueckschroehrlewiesner}.

In \cite[Thm.~1.8(1)]{cuntzkuehne:subgraphs} Cuntz and K\"uhne classified all factored members among the connected subgraph arrangements. 
For the notion of a \emph{factored} or \emph{nice} arrangement, see Definition \ref{def:factored}.

\begin{theorem}
	\label{thm:factoredAG}
	The connected subgraph arrangement $\CA_G$ is factored if and only if either $G$ is  a path graph $P_n$ or a path-with-triangle graph $\Delta_{n,1}$ for $n \ge 2$.
\end{theorem}

We strengthen Theorem \ref{thm:factoredAG} as follows.

\begin{theorem}
	\label{thm:factoredAG2}
	The arrangement $\CA_G$ is factored if and only if it is inductively factored.
\end{theorem}

For the notions of factoredness and inductive factoredness, see Definitions \ref{def:factored} and \ref{def:indfactored}.
Theorem \ref{thm:factoredAG2} is proved in
\S \ref{s:thm:factoredAG2}.

Recall the graphs from Theorem \ref{thm:freeAG}
(Definition \ref{definition: cuntz kühne graph families}).
If $G = P_n$, then $\CA_G$ is supersolvable (\cite[Thm.~1.8(2)]{cuntzkuehne:subgraphs}), so it is $K(\pi,1)$, by \cite[Prop.~5.12, Thm.~5.113]{orlikterao:arrangements}.
Moreover, $\CA_{C_3}$ is factored, so is also $K(\pi,1)$, owing to \cite{paris:factored}.
In general, $\CA_G$ is not $K(\pi,1)$.
Indeed, the following theorem shows that the graphs $G$ of possible aspherical connected subgraph arrangements $\CA_G$ are extremely restricted.

In general, $K(\pi,1)$ arrangements need not be free, e.g.~see \cite[Fig.~5.4]{orlikterao:arrangements}. However, for connected subgraph arrangements, this does seem to be the case.

\begin{theorem}
	\label{thm:kpi1}
		If $\CA_G$ is $K(\pi,1)$, then $\CA_G$ is free with the possible exception when $G = K_4$.
\end{theorem}

While $\CA_{K_4}$ is not free, it is unknown at present whether $\CA_{K_4}$ is 
$K(\pi,1)$.
Theorem \ref{thm:kpi1} is proved in \S\ref{s:thm:kpi1}.

A hyperplane arrangement is called formal provided all linear dependencies among the defining forms of the hyperplanes are generated by ones corresponding to intersections  of codimension two.
The significance of this notion stems from the fact that
complex arrangements with aspherical complements are formal, \cite[Thm.\ 4.2]{falkrandell:homotopy}.
In addition, free arrangements are known to be formal,
\cite[Cor.~2.5]{yuzvinsky:obstruction},
and factored arrangements are formal, \cite[Thm.~1.1]{moellermueckschroehre:formal}.
In our next result we show that any connected subgraph arrangement $\CA_G$ is combinatorially formal, see Definition \ref{def:combformal}.

\begin{theorem}
	\label{thm:AGformal}
	For a fixed field $\BBK$, $\CA_G(\BBK)$ is combinatorially formal.
\end{theorem}

The proof of this theorem which is presented in \S \ref{s:thm:AGformal} is based on results from \cite{moellermueckschroehre:formal}.
Note that Theorem \ref{thm:AGformal} is false for general $0/1$-arrangements, as, for instance, the generic rational  $3$-arrangement given by $Q(\CA) = xyz(x+y+z)$ is clearly not formal.

In our final section \S \ref{sec:idealAG} we investigate ideal subarrangements $\CA_\CI$ of the connected subgraph arrangements $\CA_G$ when $G$ is an almost-path graph $A_{n,k}$.
For instance, generalizing \cite[Thm.~4.1]{cuntzkuehne:subgraphs}, in Proposition \ref{prop:idealAG} we show that each ideal subarrangement $\CA_\CI$ of $\CA_{A_{n,k}}$ is also MAT-free.
We go on to determine certain subarrangements $\CA_\CI$ of $\CA_{A_{n,k}}$ which are inductively factored and $K(\pi,1)$, see Proposition \ref{prop:AGs2}.

For general information about arrangements
we refer the reader to
\cite{orlikterao:arrangements}.

\section{Preliminaries}
\label{sect:prelims}

\subsection{Hyperplane arrangements}
\label{ssect:arrangements}
Let $V = \BBC^\ell$
be an $\ell$-dimensional complex vector space.
A \emph{hyperplane arrangement} is a pair
$(\CA, V)$, where $\CA$ is a finite collection of hyperplanes in $V$.
Usually, we simply write $\CA$ in place of $(\CA, V)$.
We denote the empty arrangement in $V$ by  $\Phi_\ell$ .

The \emph{lattice} $L(\CA)$ of $\CA$ is the set of subspaces of $V$ of
the form $H_1\cap \ldots \cap H_i$ where $\{ H_1, \ldots, H_i\}$ is a subset
of $\CA$.
For $X \in L(\CA)$, we have two associated arrangements,
firstly
$\CA_X :=\{H \in \CA \mid X \subseteq H\} \subseteq \CA$,
the \emph{localization of $\CA$ at $X$},
and secondly,
the \emph{restriction of $\CA$ to $X$}, $(\CA^X,X)$, where
$\CA^X := \{ X \cap H \mid H \in \CA \setminus \CA_X\}$.
The lattice $L(\CA)$ is a partially ordered set by reverse inclusion:
$X \le Y$ provided $Y \subseteq X$ for $X,Y \in L(\CA)$.

Throughout, we only consider arrangements $\CA$
such that $0 \in H$ for each $H$ in $\CA$.
These are called \emph{central}.
In that case the \emph{center}
$T(\CA) := \cap_{H \in \CA} H$ of $\CA$ is the unique
maximal element in $L(\CA)$  with respect
to the partial order.
A \emph{rank} function on $L(\CA)$
is given by $r(X) := \codim_V(X)$.
The \emph{rank} of $\CA$
is defined as $r(\CA) := r(T(\CA))$.

Following \cite[\S 2.3]{cuntzkuehne:subgraphs},
we call a rational arrangement $\CA$  a \emph{$0/1$-arrangement} if all hyperplanes in $\CA$ are of the form $H_I = \ker\sum_{i\in I}x_i$ for some $I \subseteq [n]$. Of course, by Definition \ref{def:AG}, connected subgraph arrangements are $0/1$-arrangements.

The following families of graphs play a 
central role in the investigation of free simple connected subgraph arrangements. They are equally crucial for our study.

\begin{defn}[\cite{cuntzkuehne:subgraphs}]\label{definition: cuntz kühne graph families}
	\begin{itemize}
		\item The \emph{path graph} $P_n$ on $n$ vertices.
		\begin{figure}[H]
			\begin{tikzpicture}
				\node (t) at (-1,0) {$P_n$:};
				\node (v1) at (0,0) {};
				\node (v2) at (2,0) {};
				\node (v2h) at (2.5,0) {};
				\node (vmid) at (3,0) {};
				\node (vn-1h) at (3.5,0) {};
				\node (vn-1) at (4,0) {};
				\node (vn) at (6,0) {};
				\fill ($(v1)$) circle[radius=2pt];
				\fill ($(v2)$) circle[radius=2pt];
				\fill ($(vn-1)$) circle[radius=2pt];
				\fill ($(vn)$) circle[radius=2pt];
				\node[below] at ($(v1)$) {$1$};
				\node[below] at ($(v2)$) {$2$};
				\node[below] at ($(vn-1)$) {$n-1$};
				\node[below] at ($(vn)$) {$n$};
				\draw ($(v1)$) -- ($(v2)$) -- ($(v2h)$);
				\draw[dashed] ($(v2h)$) -- ($(vn-1h)$);
				\draw ($(vn-1h)$) -- ($(vn-1)$) -- ($(vn)$);
			\end{tikzpicture}
		\end{figure}
		\item The \emph{almost-path graph} $A_{n,k}$ on $n+1$ vertices, where $1\leq k\leq n$. Draw a path graph on $n$ vertices, add an additional vertex $n+1$, and connect the vertices $k$ and $n+1$ by an edge.
		
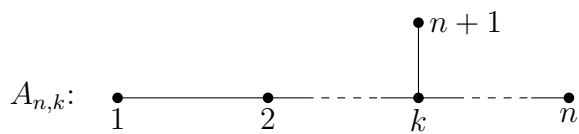
\begin{figure}[H]
			
			\begin{tikzpicture}
				\node (t) at (-1,0) {$A_{n,k}$:};
				\node (v1) at (0,0) {};
				\node (v2) at (2,0) {};
				\node (v2h+) at (2.5,0) {};
				\node (v2kmid) at (3,0) {};
				\node (vkh-) at (3.5,0) {};
				\node (vk) at (4,0) {};
				\node (vn+1) at (4,1) {};
				\node (vkh+) at (4.5,0) {};
				\node (vknmid) at (5,0) {};
				\node (vnh-) at (5.5,0) {};
				\node (vn) at (6,0) {};
				\fill ($(v1)$) circle[radius=2pt];
				\fill ($(v2)$) circle[radius=2pt];
				\fill ($(vk)$) circle[radius=2pt];
				\fill ($(vn)$) circle[radius=2pt];
				\fill ($(vn+1)$) circle[radius=2pt];
				\node[below] at ($(v1)$) {$1$};
				\node[below] at ($(v2)$) {$2$};
				\node[below] at ($(vk)$) {$k$};
				\node[below] at ($(vn)$) {$n$};
				\node[right] at ($(vn+1)$) {$n+1$};
				\draw ($(v1)$) -- ($(v2)$) -- ($(v2h+)$);
				\draw[dashed] ($(v2h+)$) -- ($(vkh-)$);
				\draw ($(vkh-)$) -- ($(vk)$) -- ($(vkh+)$);
				\draw ($(vk)$) -- ($(vn+1)$);
				\draw[dashed] ($(vkh+)$) -- ($(vnh-)$);
				\draw ($(vnh-)$) -- ($(vn)$);
			\end{tikzpicture}
		\end{figure}
		\item The \emph{cycle graph} $C_n$ on $n$ vertices. Draw $n$ vertices and connect vertex $1$ and $2$, $2$ and $3$, $\dots$, $n$ and $1$.
		
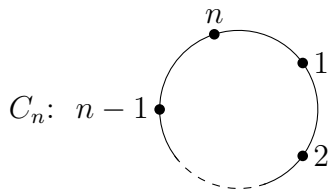
\begin{figure}[H]
			\begin{tikzpicture}
				
				\node (t) at (-2.75,0) {$C_{n}$:};
				\node (v1) at (0.85065080835203988,0.61803398874989468) {};  
				\node (v2) at (0.85065080835203988,-0.61803398874989468) {}; 
				\node (v2+) at (0.32491969623290623,-1.) {}; 
				\node (vn-1-) at (-0.85065080835203988,-0.61803398874989468) {};  
				\node (vn-1) at (-1.0514622242382672,0.0) {}; 
				\node (vn) at (-0.32491969623290623,1.) {};
				\fill ($(v1)$) circle[radius=2pt];
				\fill ($(v2)$) circle[radius=2pt];
				\fill ($(vn-1)$) circle[radius=2pt];
				\fill ($(vn)$) circle[radius=2pt];
				\node[right] at ($(v1)$) {$1$};
				\node[right] at ($(v2)$) {$2$};
				\node[left] at ($(vn-1)$) {$n-1$};
				\node[above] at ($(vn)$) {$n$};
				\draw ($(v2+)$) arc (-72:220:1.0514622242382672) ;
				\draw[dashed] ($(vn-1-)$) arc (216:290:1.0514622242382672) ;			\end{tikzpicture}
		\end{figure}
		\item The \emph{path-with-triangle graph} $\Delta_{n,k}$, where $1\leq k\leq n-1$. Draw a path graph on $n$ vertices, add an additional vertex $n+1$, and connect the vertices $k$ and $n+1$ by an edge as well as the vertices $k+1$ and $n+1$.
		
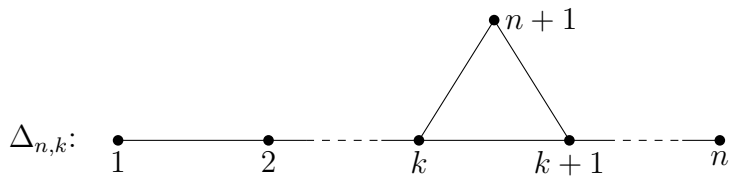
\begin{figure}[H]
			\begin{tikzpicture}
				\node (t) at (-1,0) {$\Delta_{n,k}$:};
				\node (v1) at (0,0) {};
				\node (v2) at (2,0) {};
				\node (v2h+) at (2.5,0) {};
				\node (v2kmid) at (3,0) {};
				\node (vkh-) at (3.5,0) {};
				\node (vk) at (4,0) {};
				\node (vn+1) at (5,1.6) {};
				\node (vk+1) at (6,0) {};
				\node (vk+1h+) at (6.5,0) {};
				\node (vk+1nmid) at (7,0) {};
				\node (vnh-) at (7.5,0) {};
				\node (vn) at (8,0) {};
				\fill ($(v1)$) circle[radius=2pt];
				\fill ($(v2)$) circle[radius=2pt];
				\fill ($(vk)$) circle[radius=2pt];
				\fill ($(vk+1)$) circle[radius=2pt];
				\fill ($(vn)$) circle[radius=2pt];
				\fill ($(vn+1)$) circle[radius=2pt];
				\node[below] at ($(v1)$) {$1$};
				\node[below] at ($(v2)$) {$2$};
				\node[below] at ($(vk)$) {$k$};
				\node[below] at ($(vk+1)$) {$k+1$};
				\node[below] at ($(vn)$) {$n$};
				\node[right] at ($(vn+1)$) {$n+1$};
				\draw ($(v1)$) -- ($(v2)$) -- ($(v2h+)$);
				\draw[dashed] ($(v2h+)$) -- ($(vkh-)$);
				\draw ($(vkh-)$) -- ($(vk)$) -- ($(vk+1)$) -- ($(vk+1h+)$);
				\draw ($(vk)$) -- ($(vn+1)$) -- ($(vk+1)$);
				\draw[dashed] ($(vk+1h+)$) -- ($(vnh-)$);
				\draw ($(vnh-)$) -- ($(vn)$);
			\end{tikzpicture}
		\end{figure}
	\end{itemize}
\end{defn}

\subsection{Free arrangements}
\label{ssect:free}

Free arrangements play a crucial role in the theory of arrangements;
see \cite[\S 4]{orlikterao:arrangements} for the definition and
basic properties. If $\CA$ is free, then
we can associate with $\CA$ the multiset of its \emph{exponents},
denoted $\exp \CA$. We sometimes write 
$(e_1,\ldots,e_\ell)_\le$ for the ordered set of exponents, i.e., to indicate that 
$e_1 \leq e_2 \leq \ldots \leq e_\ell$.

Terao's \emph{Factorization Theorem}
\cite{terao:freefactors} shows
that the Poincar\'e polynomial
of a free arrangement $\CA$
factors into linear terms
given by the exponents of $\CA$
(cf.\ \cite[Thm.\ 4.137]{orlikterao:arrangements}):

\begin{theorem}
	\label{thm:freefactors}
	Suppose that
	$\CA$ is free with $\exp \CA = \{ b_1, \ldots , b_\ell\}$.
	Then
	\[
	\pi(\CA,t) = \prod_{i=1}^\ell (1 + b_i t).
	\]
\end{theorem}

Terao's celebrated \emph{Addition-Deletion Theorem}
\cite{terao:freeI} plays a
fundamental role in the study of free arrangements,
\cite[Thm.\ 4.51]{orlikterao:arrangements}. The fundamental \emph{Addition Deletion Theorem} 
due to Terao  \cite{terao:freeI} plays a 
crucial role in the study of free arrangements, 
\cite[Thm.~4.51]{orlikterao:arrangements}. 
Suppose $\CA \neq \Phi_\ell$. Fix a member $H_0$ in $\CA$. Set $\CA' = \CA \setminus \{H_0\}$ and  $\CA'' = \CA^{H_0}$. Then $(\CA, \CA', \CA'')$ is frequently referred to as a \emph{triple of arrangements} (with respect to $H_0$).

\begin{theorem}
	\label{thm:add-del}
	Suppose that $\CA \ne \Phi_\ell$.
	Let  $(\CA, \CA', \CA'')$ be a triple of arrangements. Then any
	two of the following statements imply the third:
	\begin{itemize}
		\item[(i)] $\CA$ is free with $\exp \CA = \{ b_1, \ldots , b_{\ell -1}, b_\ell\}$;
		\item[(ii)] $\CA'$ is free with $\exp \CA' = \{ b_1, \ldots , b_{\ell -1}, b_\ell-1\}$;
		\item[(iii)] $\CA''$ is free with $\exp \CA'' = \{ b_1, \ldots , b_{\ell -1}\}$.
	\end{itemize}
\end{theorem}

Theorem \ref{thm:add-del} motivates the notion of an
\emph{inductively free} arrangement,
\cite[Def.\ 4.53]{orlikterao:arrangements}.

\begin{defn}
	\label{def:indfree}
	The class $\CIF$ of \emph{inductively free} arrangements
	is the smallest class of arrangements subject to
	\begin{itemize}
		\item[(i)] $\Phi_\ell \in \CIF$ for each $\ell \ge 0$;
		\item[(ii)] if there exists a hyperplane $H_0 \in \CA$ such that both
		$\CA'$ and $\CA''$ belong to $\CIF$, and $\exp \CA '' \subseteq \exp \CA'$,
		then $\CA$ also belongs to $\CIF$.
	\end{itemize}
\end{defn}

Clearly, inductively free arrangements are free. However, the latter class properly contains the former, cf.~\cite[Ex.~4.59]{orlikterao:arrangements}.

\begin{remark}
	\label{rem:indfreetable}
	It is possible to describe an inductively free arrangement $\CA$ by means of
	a so called
	\emph{induction table}, cf.~\cite[\S 4.3, p.~119]{orlikterao:arrangements}.
	In this process we start with an inductively free arrangement
	and add hyperplanes successively ensuring that
	part (ii) of Definition \ref{def:indfree} is satisfied.
	This process is referred to as \emph{induction of hyperplanes}.
	This procedure amounts to
	choosing a total order on $\CA$, say
	$\CA = \{H_1, \ldots, H_m\}$,
	so that each of the subarrangements
	$\CA_i := \{H_1, \ldots, H_i\}$
	and each of the restrictions $\CA_i^{H_i}$ is inductively free
	for $i = 1, \ldots, m$.
	In the associated induction table we record in the $i$-th row the information
	of the $i$-th step of this process, by
	listing $\exp \CA_i' = \exp \CA_{i-1}$,
	$H_i$,
	as well as $\exp \CA_i'' = \exp \CA_i^{H_i}$,
	for $i = 1, \ldots, m$.
\end{remark}

\subsection{MAT-Free arrangements}
\label{ssect:MATfree}

Next we recall the so-called Multiple Addition Theorem (MAT)  from \cite{abeetall:weyl}.

\begin{theorem}[{\cite[Thm.~3.1]{abeetall:weyl}}]
	\label{Thm_MAT}
	Let $\CA' = (\CA', V)$ be a free arrangement with
	$\exp(\CA')=(e_1,\ldots,e_\ell)_\le$
	and let $1 \le p \le \ell$ be the multiplicity of the highest exponent, i.e.,
	\[ e_1 \leq e_2 \leq \ldots \leq e_{\ell-p} < e_{\ell-p+1} =\cdots=e_\ell=:e. \]
	Let $H_1,\ldots,H_q$ be hyperplanes in $V$ with
	$H_i \not \in \CA'$ for $i=1,\ldots,q$. Define
	\[ \CA''_j:=(\CA'\cup \{H_j\})^{H_j}=\{H\cap H_{j} \mid H\in \CA'\}, \quad \text{ for }j=1,\ldots,q. \]
	Assume that the following conditions are satisfied:
	\begin{itemize}
		\item[(1)]
		$X:=H_1 \cap \cdots \cap H_q$ is $q$-codimensional.
		\item[(2)]
		$X \not \subseteq \bigcup_{H \in \CA'} H$.
		\item[(3)]
		$|\CA'|-|\CA''_j|=e$ for $1 \le j \le q$.
	\end{itemize}
	Then $q \leq p$ and $\CA:=\CA' \cup \{H_1,\ldots,H_q\}$ is free with
	$$\exp(\CA)=(e_1,\ldots,e_{\ell-q},e+1,\ldots,e+1)_\le.$$
\end{theorem}

We often consider the addition of several hyperplanes using
Theorem \ref{Thm_MAT}. This motivates the next terminology.

\begin{defn}
	\label{def:MATstep}
	Let $\CA'$ and $\{H_1,\ldots,H_q\}$ be as in Theorem \ref{Thm_MAT} such that
	conditions (1)--(3) are satisfied. Then the addition 
	of $\{H_1,\ldots,H_q\}$ to $\CA'$ resulting in $\CA = \CA' \cup \{H_1,\ldots,H_q\}$
	is called an \emph{MAT-step}.
\end{defn}

An iterative application of Theorem \ref{Thm_MAT} motivates the following natural concept.

\begin{defn}%
	[{\cite[Def.~13, Lem.~19]{Cunzmuecksch:MATfree}}]
	\label{def:MATfree}%
	An arrangement $\CA$ is called \emph{MAT-free} if there exists an ordered partition
	\[
	\pi = (\pi_1|\cdots|\pi_n)
	\] 
	of $\CA$ such that the following hold. 
	Set $\CA_0 := \varnothing_\ell$ and
	\[
	\CA_k := \bigcup_{i=1}^k \pi_i \quad\text{ for } 1 \leq k \leq n.
	\]
	Then for every $0 \leq k \leq n-1$ suppose that
	\begin{itemize}
		\item[(1)] $\rk(\pi_{k+1}) = \vert \pi_{k+1} \vert$,
		\item[(2)] $\cap_{H \in \pi_{k+1}} H \nsubseteq \bigcup_{H' \in \CA_k}H'$,
		\item[(3)] $\vert \CA_k \vert - \vert (\CA_k \cup \{H\})^H \vert = k$ for each $H \in \pi_{k+1}$,
	\end{itemize}
	i.e., $\CA_{k+1} = \CA_{k}\cup\pi_{k+1}$ is an MAT-step.
	
	An ordered partition $\pi$ with these properties is called an \emph{MAT-partition} for $\CA$.
\end{defn}

\begin{remark}
	\label{rem:MAT-free}
	Suppose that $\CA$ is MAT-free with MAT-partition
	$\pi = (\pi_1|\cdots|\pi_n)$. Then we have:
	\begin{enumerate}[(a)]
		\item 
		for each $1 \leq k \le n$, $\CA_k$ is MAT-free with MAT-partition $(\pi_1|\cdots|\pi_{k})$, 
		
		\item
		$\CA$ is free and the exponents 
		$\exp(\CA) = (e_1,\ldots,e_\ell)_\le$ of $\CA$ are given by the block sizes
		of the dual partition of $\pi$: 
		\[ 
		e_i := |\{k \mid |\pi_k|\geq \ell-i+1 \}|, 
		\]
		
		\item 
		$|\pi_1| > |\pi_2| \geq \cdots \geq |\pi_n|$.
	\end{enumerate}
\end{remark}
\begin{proof}
	Statement (a) is clear by Definition \ref{def:MATfree}.
	Statements (b) and (c) follow readily from Theorem \ref{Thm_MAT} and a simple induction.
\end{proof}

\subsection{Accurate arrangements}
\label{ssect:accuracy}

We also consider the following stronger notion of freeness from \cite{mueckschroehrle:accurate}.

\begin{defn}
\label{def:accurate}
 A free arrangement $\CA$ with $\exp(\CA)= (e_1,\ldots,e_\ell)_\le$ is \emph{accurate}, if for every $1\leq d \leq \ell$, there exists an intersection $X_d\in L(\CA)$ of dimension $d$ such that the restriction $\CA^{X_d}$ is free with $\exp(\CA^{X_d}) =  (e_1,\ldots,e_d)_\le$. 
\end{defn}

The main theorem in \cite[Thm.~1.2]{mueckschroehrle:accurate}
shows that MAT-freeness implies \emph{accuracy}.
It is also known that the braid arrangement $\CA_G$ for $G$ a path graph is accurate, \cite[Thm.~1.6]{mueckschroehrle:accurate}.
In \cite[Thm.~4.1]{cuntzkuehne:subgraphs}, the authors show that $\CA_G$ is MAT-free for $G = A_{n,k}$.

\subsection{Nice arrangements}
\label{ssect:factored}

The notion of a \emph{nice} or \emph{factored}
arrangement goes back to Terao \cite{terao:factored}.
It generalizes the concept of a supersolvable arrangement.
We recall the relevant notions and results from \cite{terao:factored}
(cf.\  \cite[\S 2.3]{orlikterao:arrangements}).

\begin{defn}
	\label{def:independent}
	Let $\pi = (\pi_1, \ldots , \pi_s)$ be a partition of $\CA$.
	Then $\pi$ is called \emph{independent}, provided
	for any choice $H_i \in \pi_i$ for $1 \le i \le s$,
	the resulting $s$ hyperplanes are linearly independent, i.e.,
	$r(H_1 \cap \ldots \cap H_s) = s$.
\end{defn}

\begin{defn}
	\label{def:indpart}
	Let $\pi = (\pi_1, \ldots , \pi_s)$ be a partition of $\CA$
	and let $X \in L(\CA)$.
	The \emph{induced partition} $\pi_X$ of $\CA_X$ is given by the non-empty
	blocks of the form $\pi_i \cap \CA_X$.
\end{defn}

\begin{defn}
	\label{def:factored}
	The partition
	$\pi$ of $\CA$ is
	\emph{nice} for $\CA$ or a \emph{factorization} of $\CA$  provided
	\begin{itemize}
		\item[(i)] $\pi$ is independent, and
		\item[(ii)] for each $X \in L(\CA) \setminus \{V\}$, the induced partition $\pi_X$ admits a block
		which is a singleton.
	\end{itemize}
	If $\CA$ admits a factorization, then we also say that $\CA$ is \emph{factored} or \emph{nice}.
\end{defn}

\begin{remark}
	\label{rem:factored}
	The class of nice arrangements is closed under taking localizations.
	For, if $\CA$ is non-empty and
	$\pi$ is a nice partition of $\CA$, then the non-empty parts of the
	induced partition $\pi_X$ form a nice partition of $\CA_X$
	for each $X \in L(\CA)\setminus\{V\}$;
	cf.~the proof of \cite[Cor.~2.11]{terao:factored}.
\end{remark}

We recall the main results from \cite{terao:factored}
(cf.\  \cite[\S 3.3]{orlikterao:arrangements}) that motivated
Definition \ref{def:factored}.
Let $A(\CA)$ be the \emph{Orlik-Solomon algebra}, introduced by Orlik and Solomon in \cite{orliksolomon:hyperplanes}, see also \cite{orlikterao:arrangements}.
Let $\pi = (\pi_1, \ldots, \pi_s)$ be a partition of $\CA$ and let
\[
[\pi_i] := \BBK + \sum_{H\in \pi_i} \BBK a_H
\]
be the $\BBK$-subspace of $A(\CA)$
spanned by $1$ and the set of $\BBK$-algebra generators $a_H$
of $A(\CA)$ corresponding to the members in $\pi_i$.
So the Poincar\'e polynomial of the graded $\BBK$-vector space $[\pi_i]$ is just
$\Poin([\pi_i],t) = 1 + | \pi_i| t$.
Consider the canonical $\BBK$-linear map
\begin{equation}
	\label{eq:factoredOS}
	\kappa : [\pi_1] \otimes \cdots \otimes [\pi_s] \to A(\CA)
\end{equation}
given by multiplication. We say that $\pi$ gives
rise to a \emph{tensor factorization} of
$A(\CA)$
if $\kappa$ is an isomorphism of graded $\BBK$-vector spaces. In this case
$s = r$, as $r$ is the top degree of $A(\CA)$, and thus we get
a factorization of the Poincar\'e polynomial of $A(\CA)$ into linear terms
\begin{equation}
	\label{eq:poinOS}
	\Poin(A(\CA),t) = \prod_{i=1}^r (1 + |\pi_i| t).
\end{equation}
For $\CA = \Phi_\ell$ the empty arrangement,
we set $[\varnothing] := \BBK$, so that
$\kappa : [\varnothing] \cong A(\Phi_\ell)$.

In \cite[Thm.\ 5.3]{orliksolomonterao:hyperplanes},
Orlik, Solomon and Terao showed that
a supersolvable arrangement $\CA$ admits a partition
$\pi$  which gives rise to a tensor factorization of $A(\CA)$
via $\kappa$  in \eqref{eq:factoredOS}
(cf.\ \cite[Thm.\ 3.81]{orlikterao:arrangements}).

In \cite[Thm.\ 2.8]{terao:factored},
Terao proved that
$\pi$ gives
rise to a tensor factorization of the Orlik-Solomon algebra $A(\CA)$
via $\kappa$
as in \eqref{eq:factoredOS}
if and only if
$\pi$ is
nice for
$\CA$, see Theorem \ref{thm:teraofactored}
(cf.\ \cite[Thm.\ 3.87]{orlikterao:arrangements}).
Note that $\kappa$ is not an isomorphim of $\BBK$-algebras.

\begin{theorem}
	\label{thm:teraofactored}
	Let $\CA$ be a central $\ell$-arrangement and let
	$\pi = (\pi_1, \ldots, \pi_s)$ be a partition of $\CA$.
	Then the $\BBK$-linear map $\kappa$
	defined in \eqref{eq:factoredOS}
	is an isomorphism of graded $\BBK$-vector spaces
	if and only if $\pi$ is nice for $\CA$.
\end{theorem}

\begin{corollary}
	\label{cor:teraofactored}
	Let  $\pi = (\pi_1, \ldots, \pi_s)$ be a factorization of $\CA$.
	Then the following hold:
	\begin{itemize}
		\item[(i)] $s = r = r(\CA)$ and
		\[
		\Poin(A(\CA),t) = \prod_{i=1}^r (1 + |\pi_i|t);
		\]
		\item[(ii)]
		the multiset $\{|\pi_1|, \ldots, |\pi_r|\}$ only depends on $\CA$;
		\item[(iii)]
		for any $X \in L(\CA)$, we have
		\[
		r(X) = |\{ i \mid \pi_i \cap \CA_X \ne \varnothing \}|.
		\]
	\end{itemize}
\end{corollary}

	The following is immediate from
		Corollary \ref{cor:teraofactored} and Theorem \ref{thm:freefactors}.

	\begin{lemma}
		\label{lem:FactoredFree}
		Let $(\CA,\pi)$ be a factored arrangement which is also free.
		Then $\exp{\CA} = \{|\pi_1|,\ldots,|\pi_\ell|\}$.
	\end{lemma}

\subsection{Inductively factored arrangements}
\label{sec:indfactored}

Following Jambu and Paris
\cite{jambuparis:factored} and \cite{hogeroehrle:factored},
we introduce further notation.
Suppose that $\CA$ is non-empty and
let $\pi = (\pi_1, \ldots, \pi_s)$ be a  partition  of $\CA$.
Let $H_0 \in \pi_1$ and let
$(\CA, \CA', \CA'')$ be the triple associated with $H_0$.
We have the \emph{induced partition}
$\pi'$ of $\CA'$
consisting of the non-empty parts $\pi_i' := \pi_i \cap \CA'$.
Further, we have the \emph{restriction map}
$\R = \R_{\pi,H_0} : \CA \setminus \pi_1 \to \CA''$ given by
$H \mapsto H \cap H_0$, depending on $\pi$ and $H_0$.
Let $\pi_i'' := \R(\pi_i)$ for $i = 2, \ldots, s$.
Clearly, imposing that
$\pi'' = (\pi''_2, \ldots, \pi''_s)$ is
again a partition of $\CA''$ entails that
$\R$ is onto.

Here is the analogue for nice arrangements of Terao's
Addition-Deletion Theorem (cf.~Theorem \ref{thm:add-del}) for free arrangements from
\cite{hogeroehrle:factored}.

\begin{theorem}  
	\label{thm:add-del-factored}
	Suppose 
	$\pi = (\pi_1, \ldots, \pi_s)$ is a  partition  of $\CA  \ne \Phi_\ell$.
	Let 
	$(\CA, \CA', \CA'')$ be the triple associated with $H_0  \in \pi_1$.
	Then any two of the following statements imply the third:
	\begin{itemize}
		\item[(i)] $\pi$ is nice for $\CA$;
		\item[(ii)] $\pi'$ is nice for $\CA'$;
		\item[(iii)] $\R: \CA \setminus \pi_1 \to \CA''$
		is bijective and $\pi''$ is nice for $\CA''$.
	\end{itemize}
\end{theorem}

\begin{defn}
	\label{def:distinguished}
	Suppose $\CA \ne \Phi_\ell$. 
	Let $\pi = (\pi_1, \ldots, \pi_s)$ be a partition of $\CA$.
	Let $H_0 \in \pi_1$ and
	let $(\CA, \CA', \CA'')$ be the triple associated with $H_0$.
	We say that $H_0$ is \emph{distinguished (with respect to $\pi$)}
	provided $\pi$ induces a factorization $\pi'$ of
	$\CA'$, i.e., the non-empty
	subsets $\pi_i \cap \CA'$ form a nice partition
	of $\CA'$. Note that since $H_0 \in \pi_1$, we have
	$\pi_i \cap \CA' = \pi_i\not= \varnothing$
	for $i = 2, \ldots, s$.
\end{defn}

The Addition-Deletion Theorem \ref{thm:add-del-factored}
for nice arrangements motivates
the following stronger notion of factorization,
cf.\ \cite{jambuparis:factored}.

\begin{defn} [{\cite[Def.~3.8]{hogeroehrle:factored}}]
	\label{def:indfactored}
	The class $\CIFAC$ of \emph{inductively factored} arrangements
	is the smallest class of pairs $(\CA, \pi)$ of
	arrangements $\CA$ along with a partition $\pi$
	subject to
	\begin{itemize}
		\item[(i)] $(\Phi_\ell, (\varnothing)) \in \CIFAC$ for each $\ell \ge 0$;
		\item[(ii)] if there exists a partition $\pi$ of $\CA$
		and a hyperplane $H_0 \in \pi_1$ such that
		for the triple $(\CA, \CA', \CA'')$ associated with $H_0$
		the restriction map $\R = \R_{\pi, H_0} : \CA \setminus \pi_1 \to \CA''$
		is bijective and for the induced partitions $\pi'$ of $\CA'$ and
		$\pi''$ of $\CA''$
		both $(\CA', \pi')$ and $(\CA'', \pi'')$ belong to $\CIFAC$,
		then $(\CA, \pi)$ also belongs to $\CIFAC$.
	\end{itemize}
	If $(\CA, \pi)$ is in $\CIFAC$, then we say that
	$\CA$ is \emph{inductively factored with respect to $\pi$}, or else
	that $\pi$ is an \emph{inductive factorization} of $\CA$.
	Usually, we say $\CA$ is \emph{inductively factored} without
	reference to a specific inductive factorization of $\CA$.
\end{defn}

\begin{remark}
	Definition \ref{def:indfactored}
	of inductively factored arrangements
	differs from the one given by
	Jambu and Paris \cite{jambuparis:factored}
	in that, apart from the mere technicalities of
	incorporating empty arrangements and for defining
	$\CIFAC$ for pairs of arrangements and partitions
	rather than for partitions of arrangements,
	in part (ii) we do not assume from the outset
	that $\pi$ is a factorization of $\CA$.
	This is possible by virtue of
	the Addition-Deletion Theorem \ref{thm:add-del-factored}.
	However, this comes at the cost of
	the bijectivity requirement for the associated restriction map $\R$.
\end{remark}

\begin{remark}
	\label{rem:indtable}
	(i).
	If $\CA$ is inductively factored, then  $\CA$ is inductively free,
	by \cite[Prop.~3.14]{hogeroehrle:factored}.
	The latter can be described by an induction table, see Remark \ref{rem:indfreetable}.
	The proof of \cite[Prop.~3.14]{hogeroehrle:factored} shows
	that if $\pi$ is an inductive factorization of $\CA$
	and $H_0 \in \CA$ is distinguished with respect to $\pi$, then
	the triple $(\CA, \CA', \CA'')$ with respect to $H_0$
	is a triple of inductively free arrangements.
	Thus an induction table of $\CA$ can be constructed,
	compatible with suitable inductive factorizations of
	the subarrangements $\CA_i$.

	(ii).
	Now suppose $\CA$ is inductively free and let
	$\CA = \{H_1, \ldots, H_n\}$ be a
	choice of a total order on $\CA$,
	so that each of the subarrangements
	$\CA_0 := \Phi_\ell$, $\CA_i := \{H_1, \ldots, H_i\}$
	and each of the restrictions $\CA_i^{H_i}$ is inductively free
	for $i = 1, \ldots, n$.

	Then, starting with the empty partition for $\Phi_\ell$, we can attempt to
	build inductive factorizations $\pi_i$ of $\CA_i$
	consecutively, resulting in an inductive factorization
	$\pi = \pi_n$ of $\CA = \CA_n$.
	This is achieved by invoking Theorem \ref{thm:add-del-factored}
	repeatedly in order to derive that each $\pi_i$ is an inductive factorization of $\CA_i$.
	For this it suffices to check
	the conditions in part (iii) of Theorem \ref{thm:add-del-factored}, i.e.,  that
	$\exp \CA_i''$ is given by the sizes of the parts of $\pi_i$ not containing $H_i$
	and that the induced partition $\pi_i''$ of $\CA_i''$ is a factorization.
	The fact that $H_i$ is distinguished
	with respect to $\pi_i$ is part of the inductive hypothesis, as
	$\pi_i' = \pi_{i-1}$ is an inductive factorization of $\CA_i' = \CA_{i-1}$.

	We then add the inductive factorizations $\pi_i$
	of $\CA_i$ as additional data into an induction table for $\CA$
	(or else record to which part of $\pi_{i-1}$ the new hyperplane $H_i$ is appended to).
	The data in such an extended induction table
	together with the ``Addition'' part of Theorem \ref{thm:add-del-factored}
	then proves that $\CA$ is
	inductively factored.
	We refer to this technique as \emph{induction of factorizations} and the
	corresponding table as an \emph{induction table of factorizations} for $\CA$.

	We illustrate this induction of factorizations procedure in Tables \ref{table0} and \ref{table1}.
	As for the usual induction process for free arrangements,
	induction for factorizations is sensitive to the chosen order on $\CA$.
\end{remark}

\subsection{$K(\pi,1)$-arrangements}
\label{ssect:kpionearrangements}
A complex $\ell$-arrangement $\CA$ is called  \emph{aspherical}, or a
\emph{$K(\pi,1)$-arrangement} (or that $\CA$ is $K(\pi,1)$ for short), provided
the complement $\CM(\CA)$ of the union of the hyperplanes in
$\CA$ in $\BBC^\ell$ is aspherical, i.e.,  is a
$K(\pi,1)$-space. That is, the universal covering space of $\CM(\CA)$
is contractible and the fundamental group
$\pi_1(\CM(\CA))$ of $\CM(\CA)$ is isomorphic to the group $\pi$.
This is an important
topological property, for
the cohomology ring $H^*(X, \BBZ)$ of a $K(\pi,1)$-space $X$
coincides
with the group cohomology $H^*(\pi, \BBZ)$ of $\pi$.
The crucial point here is that the intersections of codimension $2$
determine the fundamental group $\pi_1(\CM(\CA))$ of $\CM(\CA)$.

By Deligne's seminal result \cite{deligne}, complexified simplicial arrangements are $K(\pi, 1)$.
Likewise for complex supersolvable arrangements,
cf.~\cite{falkrandell:fiber-type} and \cite{terao:modular}
(cf.~\cite[Prop.\ 5.12, Thm.~5.113]{orlikterao:arrangements}).

\subsection{Local Properties}
\label{sub:local}
A property for arrangements is called \emph{local} if it is preserved under localizations, \cite[Def.~6.1]{cuntzkuehne:subgraphs}.
Typical examples of local properties of arrangements are freeness \cite[Thm.~4.37]{orlikterao:arrangements},
inductive freeness \cite[Thm~1.1]{hogeroehrleschauenburg:free}, factoredness \cite[Cor.~3.90]{orlikterao:arrangements}, and
asphericity, see Remark \ref{rem:local}.
We record two graph theoretic constructions from
\cite{cuntzkuehne:subgraphs} which allow us to
control local properties of the $\CA_G$.

Let $G = (N, E)$ be a graph.	Let $S \subset N$ be a subset of nodes and let $G[S]$ be the \emph{induced subgraph on $S$}. Then $\CA_{G[S]}$ is a localization of $\CA_G$; cf.~proof of \cite[Lem.~6.2]{cuntzkuehne:subgraphs}.
Let $e \in  E$ be an edge. Denote by $G/e$
the \emph{graph with contracted edge $e$}, i.e., the graph in which the vertices of $e$ are identified to
a single vertex, and multiple edges or loops are discarded. Then $\CA_{G/e}$ is a localization of $\CA_G$; cf.~proof of \cite[Lem.~6.4]{cuntzkuehne:subgraphs}.

\begin{lemma}[{\cite[Lem.~6.2, Lem.~6.4]{cuntzkuehne:subgraphs}}]
	\label{lem:local}
	Let $G = (N, E)$ be a graph. Assume
	the connected subgraph arrangement $\CA_G$ has the local arrangement property $P$.
	\begin{itemize}
		\item [(i)]
			Let $S \subset N$ be a subset of nodes. Then $\CA_{G[S]}$ also has property $P$.
		\item [(ii)]
		Let $e \in E$ be an edge of $G$. Then $\CA_{G/e}$ also has property $P$.
	\end{itemize}
\end{lemma}

\section{Projective Uniqueness: Proof of Theorem \ref{thm:AGiscomb}}
\label{s:thm:AGiscomb}

\begin{defn}
	\label{def:projunique}
	Let $\CA$ and $\CB$ be two arrangements in a $\BBK$-vector space $V$.
	\begin{itemize}
		\item[(i)] $\CA$ and $\CB$ are \emph{linearly isomorphic} if there is a $\varphi \in \GL(V)$
		such that $\CB =  \{\varphi(H) \mid H \in \CA\}$; denoted by $\CA \cong \CB$.
		\item[(ii)] $\CA$ and $\CB$ are \emph{$L$-equivalent} if $L(\CA)$ and $L(\CB)$ are isomorphic as posets;
			denoted by $\CA \cong_L \CB$.
		\item[(iii)] $\CA$ is \emph{projectively unique} if for any arrangement $\CC$ in $V$ we have:
			$\CC \cong_L \CA$ implies $\CC \cong \CA$.
	\end{itemize}
\end{defn}

\begin{remark}
	\label{rem:ProjUniqe}
	The term ``projectively unique'' originates from the dual point of view.
	Every hyperplane arrangement $\CA$ in $V$ yields a dual point configuration in the projective space $\BBP(V)$,
	denoted by $\CA^*$, see \cite[\S 2]{ziegler:matroid}.
	Two arrangements $\CA$, $\CB$ in $V$ are linearly isomorphic if and only if there is a projective transformation $\varphi^*$
	which maps $\CA^*$ to $\CB^*$.
\end{remark}

\begin{example}
	\label{ex:Proj(non)unique}
	(i). Let $\CA$ be a generic arrangement in $V = \BBK^\ell$ with $|\CA| = \ell+1$.
			Then $\CA = \{H_1,\ldots,H_\ell,H'\}$ where $\alpha_1,\ldots,\alpha_\ell \in V^*$
			with $H_i = \ker(\alpha_i)$ and $\alpha_1,\ldots,\alpha_\ell$ form a basis of $V^*$.
			If $\gamma = \sum_{i=1}^\ell a_i\alpha_i \in V^*$ with $\ker(\gamma) = H'$, then by the genericity of $\CA$
			we have $a_i\neq 0$ for all $1\leq i \leq \ell$. Hence, after rescaling the $\alpha_i$,
			we see that $\CA$ is linearly isomorphic to $\{\ker(x_1),\ldots,\ker(x_\ell),\ker(\sum_{i=1}^\ell x_i)\}$,
			where $x_1,\ldots,x_\ell$ is some basis of $V^*$.
			Thus, $\CA$ is projectively unique.
			
		(ii). Let $V = \BBQ^2$ and $x,y$ be a basis of $V^*$.
			Let $\CA = \{\ker(x),\ker(y),\ker(x+y),\ker(x+2y)\}$ and $\CB = \{\ker(x),\ker(y),\ker(x+y),\ker(x+3y)\}$ be arrangements in $V$.
			Then clearly $\CA \cong_L \CB$ but $\CA$ is not linearly isomorphic to $\CB$.
			Thus, $\CA$ (resp.\ $\CB$) is not projectively unique.
\end{example}

Fix a field $\BBK$.
Let $\CA$ be a $\BBK$-arrangement.
We specify what we mean by a subarrangement of $\CA$ being generated by a subset (or subarrangement) of $\CA$ 
(cf.\ \cite[Def.~3.4]{Cun21_Greedy} for an equivalent formulation for arrangements of rank $3$).

\begin{defn}
	\label{def:gen}
	Let $\varnothing \ne S \subseteq \CA$. Set $\Gen_0(\CA,S) := S$ and inductively $$\Gen_{i+1}(\CA,S) := \left\{H \in \CA \mid \exists\  J \subseteq L(\Gen_i(\CA,S)) : H = \sum_{X \in J} X \right\}$$ for $i \ge 0$.
	Then we say that
	$$\langle S \rangle_\CA := \bigcup_{i\ge 0} \Gen_i(\CA,S) \subseteq \CA$$
	is the subarrangement of $\CA$
	\emph{generated by $S$}. If $\langle S \rangle_\CA  = \CA$, then we say that $S$ generates $\CA$. 
\end{defn}

\begin{lemma}
	\label{lem:gen}
	Let $\CA$ and $\CB$ be two arrangements in $V$.
	Suppose $\varnothing \ne S \subseteq \CA$, $\varnothing \ne T \subseteq \CB$
	such that $\langle S \rangle_\CA = \CA$ and $\langle T \rangle_\CB = \CB$, i.e., $S$ generates $\CA$
	and $T$ generates $\CB$.
	If $\CA$ and $\CB$ are $L$-equivalent via a poset isomorphism $\psi:L(\CA)\to L(\CB)$ and $\varphi \in \GL(V)$
	such that $\psi(S) = \varphi(S) = T$ and $\psi(H) = \varphi(H)$ for all $H \in S$, 
	then $\varphi$ extends to a linear isomorphism between the whole arrangements, i.e., 
	$\varphi(\CA) = \CB$.
\end{lemma}

\begin{proof}
	Set $G_i := \Gen_i(\CA,S)$, and $G_i' := \Gen_i(\CB,T)$.
	We argue by induction that $\varphi(G_i) = G_i'$ for all $i\geq 0$ which implies the statement.
	By assumption, the statement is clear for $i=0$.
	Let $H' \in G_{i+1}'$. Then $H' = \sum_{X'\in J'}X'$ for some $J' \subseteq L(G_i')$.
	But $L(G_i') = \psi(L(G_i))$ and by the induction hypothesis $G_i' = \psi(G_i) = \varphi(G_i)$.
	Thus, with $J = \psi^{-1}(J') = \varphi^{-1}(J')  \subseteq L(G_i)$,
	for $H = \sum_{X \in J}X \in G_{i+1}$ we have $\varphi(\sum_{X\in J}X) = \sum_{X' \in J'}X' = H'$. 
\end{proof}

\begin{prop}
	\label{prop:gen}
	Let $\CA$ be an essential and irreducible arrangement in $V \cong \BBK^\ell$.
	Suppose there is a subset $S$ of $\CA$ such that $\langle S \rangle_\CA = \CA$ and $|S| = \ell+1$.
	Then $\CA$ is projectively unique.
\end{prop}
\begin{proof}
	Since $\CA$ is irreducible, after an appropriate coordinate change we may assume without loss of generality that $S = \{H_1=\ker(x_1),\ldots,H_\ell = \ker(x_\ell),H_{\ell+1}=\ker(\sum_{i=1}^\ell x_i)\}$
	where $x_1,\ldots,x_\ell$ is a basis for $V^*$ (cf.\ Example \ref{ex:Proj(non)unique}(i)).
	Let $\CB$ be another arrangement in $V$ which is $L$-equivalent to $\CA$ via $\psi:L(\CA)\to L(\CB)$ and set $T := \psi(S)$.
	Since $\psi$ is a poset isomorphism, we have $\langle T \rangle_\CB = \CB$.
	Moreover, $T$ spans a generic subarrangement of $\CB$ with $\ell+1$ hyperplanes, 
	say $T = \{\psi(H_1)= \ker(\beta_1),\ldots,\psi(H_\ell) = \ker(\beta_\ell),\psi(H_{\ell+1}) = \ker(\gamma)\}$.
	After rescaling $\beta_1,\ldots,\beta_\ell$ by suitable scalars, we can assume $\gamma=\sum_{i=1}^\ell\beta_i$
	and note that $\beta_1,\ldots,\beta_\ell$ yield a basis for $V^*$. Thus, if we take $\varphi \in \GL(V)$ 
	as the unique linear map which is dual to the transformation of $V^*$ which maps $x_1,\ldots,x_\ell$ to $\beta_1,\ldots,\beta_\ell$,
	we have $\psi(S) = \varphi(S) = T$ and $\varphi(H) = \psi(H)$ for all $H\in S$. 
	By Lemma \ref{lem:gen}, $\CB$ is linearly isomorphic to $\CA$ which concludes the proof.
\end{proof}

Finally, 
Theorem \ref{thm:AGiscomb} is an immediate consequence of our next result.

\begin{prop}
	\label{prop:genCSG}
	Let $G$ be a connected graph on $[n]$ and 
	let $S := \{\ker x_1, \ldots, \ker x_n, \ker(x_1 + \ldots + x_n)\} \subseteq \CA_G$. Then
	 $\langle S \rangle_{\CA_G} = \CA_G$.
	 In particular, $\CA_G$ is projectively unique over $\BBQ$.
\end{prop}

\begin{proof}
	Let $\varnothing \ne I \subseteq [n]$ such that the induced subgraph $G[I]$ of $G$ on $I$ is connected, i.e., $H_I = \ker(\sum_{i\in I}x_i) \in \CA_G$.
	Then define
	\begin{align*}
	X & := \bigcap_{i \in I}\ker(x_i) \text{ and } \\
	Y & := \ker(x_1 + \ldots + x_n) \cap \bigcap_{j \in [n]\setminus I}\ker(x_j).
\end{align*}
Both $X$ and $Y$ belong to the lattice of intersections of the subarrangement $S$ of $\CA_G$, $\dim X = n - |I|$, $\dim Y = |I|-1$, $\dim (X \cap Y) = 0$, and $X, Y \subseteq H_I$.
It follows that $H_I = X + Y \in \langle S \rangle_{\CA_G}$. Since $\varnothing \ne I \subseteq [n]$ is arbitrary such that the induced subgraph $G[I]$ of $G$ on $I$ is connected, the result follows thanks to Proposition \ref{prop:gen}.
\end{proof}

\begin{remark}
	\label{rem:gen}
	Proposition \ref{prop:gen} applies more generally to $0/1$-arrangements $\CA$ over any field 
	which satisfy the condition of Proposition \ref{prop:genCSG}. Also note that, by assumption, 
	such $\CA$ contain the generic subarrangement of full rank consisting of the hyperplanes in $S$.
	Thus $\CA$ is necessarily irreducible.
\end{remark}

\section{Accuracy: Proof of Theorem \ref{thm:accurateAG}}
\label{s:accurateAG}

In this section we characterize accuracy for connected subgraph arrangements. The arrangement $\CA_G$ is free if and only if $G$ is a path, almost-path, path-with-triangle, or cycle graph. All of these arrangements are known to be accurate, the only exception being the path-with-triangle graph. The path graph and the cycle graph are isomorphic to the braid arrangement and the cone over the Shi arrangement, respectively. The braid arrangement is accurate, see \cite[Thm.~1.6]{mueckschroehrle:accurate}. Also, the cone over the Shi arrangement was shown to be accurate in \cite[Thm.~1.8]{mueckschroehrle:accurate}. The almost-path graph is MAT-free, and hence accurate due to \cite[Thm.~1.2]{mueckschroehrle:accurate}. We are left to show that $\CA_{\Delta_{n,k}}$ is accurate. We first need to study the automorphism group of the almost-path graph arrangement.
\begin{remark}
	 The arrangement $\CA_{A_{n,k}}$ contains three independent path graph arrangements $\CA_{P_1}$,$\CA_{P_{k-1}}$, and $\CA_{P_{n-k}}$ corresponding to each of the three branches of the almost-path graph. It follows that the automorphism group of the arrangement contains the product of the automorphism groups of these three path graph arrangements, namely $\mathfrak{S}_2\times\mathfrak{S}_k\times\mathfrak{S}_{n-k+1}$. 
   Denoting by $\Aut(A_{n,k})$ the automorphism group of the graph $A_{n,k}$ we consider the subgroup  
   \begin{equation}
       \label{eq:ank}
          \mathcal G := \Aut(A_{n,k})\ltimes\mathfrak{S}_2\times\mathfrak{S}_{k-1}\times\mathfrak{S}_{n-k+1}
   \end{equation}
  of the automorphism group of $\CA_{A_{n,k}}$.
  More precisely, we obtain the following groups depending on $n$ and $k$:
	 \begin{itemize}
	 	\item $\Aut(A_{3,2})\ltimes(\mathfrak{S}_2\times\mathfrak{S}_2\times\mathfrak{S}_2)$ for $(n,k)=(3,2)$,
	 	\item $(\Aut(A_{n,2})\ltimes(\mathfrak{S}_2\times\mathfrak{S}_2))\times\mathfrak{S}_{n-2}$ for $(n,k)=(n,2)$,
	 	\item $\mathfrak{S}_2\times (\Aut(A_{2k-1,k})\ltimes(\mathfrak{S}_k\times\mathfrak{S}_k))$ for $(n,k)=(2k-1,k)$,
	 	\item $\mathfrak{S}_2\times\mathfrak{S}_k\times\mathfrak{S}_{n-k+1}$ otherwise (here $\Aut(A_{n,k})$ is trivial).
	 \end{itemize}
\end{remark}

From the description above, we deduce that the action of the group $\mathcal G$ from \eqref{eq:ank}
on $\CA_{A_{n,k}}$ has exactly $4-g$ orbits, where $g$ is the minimal number of generators of $\Aut(A_{n,k})$. In the general case when $\Aut(A_{n,k}) = 1$, there are
\begin{enumerate}
	\item three orbits each corresponding to a path graph branch of $A_{n,k}$,
	\item one orbit consisting of all hyperplanes $H_I$ with $k\in I$.
\end{enumerate}
When $\Aut(A_{n,k})$ is non-trivial, some of the first three orbits are merged together, namely, those corresponding to branches of the same length.

We now prove accuracy for $\CA_{\Delta_{n,k}}$. Consider the connected subgraph arrangement $\CA_{A_{n+1,k+1}}$ with its set of ordered exponents $(e_1, e_2, \ldots, e_{n+2})_\leq$, and its MAT-partition $\pi = (\pi_1, \ldots,\pi_{n+2})$, where $\pi_d = \{H_I \mid |I| = d\}$. For every hyperplane $H$ in the same orbit of $x_{k+1} = 0$, i.e., for $H=H_I$ with $k+1\in I$, we have $$\CA_{\Delta_{n,k}} = \CA_{A_{n+1,k+1}}^{H}.$$ Because $\CA_{A_{n+1,k+1}}$ is MAT-free,  \cite[Thm.~3.9]{mueckschroehrle:accurate} implies that for each $1\leq k \leq n$ and every $|\pi_{k+1}|\leq q\leq |\pi_k|$, there exists a subset $\CC\subseteq \pi_k$ with $|\CC| = q$ such that, for $X = \bigcap_{H\in\CC} H$, of dimension $n+2-q$, the restriction $\CA_{A_{n+1,k+1}}^X$ is free with exponents $(e_1, e_2, \ldots, e_{n+2-q})_\leq$. This is how the accuracy of MAT-free arrangements is proven in general. Now we observe that for some $H\in\pi_i\setminus\CC$, we still have that $\CA_{A_{n+1,k+1}}^{X\cap H}$ is free with exponents $(e_1, e_2, \ldots, e_{n+1-q})_\leq$. Because an element $H_I$ with $k+1\in I$ exists in every block of the MAT-partition $\pi$ of $\CA_{A_{n+1,k+1}}$, we can always choose such an element for every choice of $X$. Finally, we observe that $\exp(\CA_{\Delta_{n,k}})=(e_1,  e_2, \ldots, e_{n+1})_\leq \subseteq (e_1, \ e_2, \ldots, e_{n+2})_\leq = \exp(\CA_{A_{n+1,k+1}})$. Therefore, we can conclude that
$$
\CA_{\Delta_{n,k}}^X = (\CA_{A_{n+1,k+1}}^H)^X = \CA_{A_{n+1,k+1}}^{X\cap H}
$$
is again free with the right exponents, that is, $\CA_{\Delta_{n,k}}$ is accurate. \qed

There is a natural stronger notion of accuracy as in Definition \ref{def:accurate}. We say that the free arrangement $\CA$ is \emph{flag-accurate}, if $\CA$ is accurate and if the tuple 
of flats $(X_1, X_2, \ldots, X_\ell)$ from Definition \ref{def:accurate} can be chosen to be a flag in $L(\CA)$, see \cite[Def.~1.1]{mueckschroehrletran:flagaccurate}.
It is shown in \cite[Cor.~1.7]{mueckschroehrletran:flagaccurate} that Coxeter arrangements are flag-accurate and in \cite[Thm.~1.14]{mueckschroehrletran:flagaccurate} that all Shi arrangements are also flag-accurate. Consequently, the connected subgraph arrangements $\CA_G$ are flag-accurate provided $G$ is a path graph or a cycle graph.
It follows from 
\cite[Lem.~3.1]{mueckschroehrletran:flagaccurate} and \cite[Thm.~4.1, Prop.~5.1, Thm.~5.4]{cuntzkuehne:subgraphs}
that $\CA_{A_{n+1,k+1}}$ is flag-accurate if and only if 
$\CA_{\Delta_{n,k}}$ is flag-accurate 
for all $n$ and $k$.
So potentially, this gives an inductive means to derive this property also for almost-path graphs and path-with-triangle graphs.
Moreover, it is shown in \cite[Thm.~4.8]{mueckschroehrletran:flagaccurate} that ideal arrangements up to rank $8$ are also flag-accurate. Thus $\CA_G$ is flag-accurate 
for $ G = A_{n,k}$ for $3 \le n \le 8$ and $2\le k \le 3$. 
In view of these results and of Theorem \ref{thm:accurateAG}, we 
put forward the following.

\begin{conjecture}
	\label{conj:flag-accAG}
	$\CA_G$ is free if and only if it is flag-accurate.
\end{conjecture}
 
 It is clear that flag-accuracy is a combinatorial property which only depends on the lattice of intersections. This is not known for 
 accuracy, though; see \cite[Rem.~1.3]{mueckschroehrletran:flagaccurate}.

\section{Inductive Freeness: Evidence for Conjecture \ref{conj:freeAG2}}
\label{s:conj:freeAG2}

According to Theorem \ref{thm:freeAG},
$\CA_G$ is free if and only if $G$ is
a path, a cycle, an almost path, or a path-with-triangle graph.
For $G$ a path graph, $\CA_G$ is supersolvable, by \cite[Cor.~8.11]{cuntzkuehne:subgraphs}, and supersolvable arrangements are inductively free,
thanks to work of Jambu and Terao \cite[Thm.~4.2]{jambuterao:free}.
For $G$ a cycle,
Cuntz and K\"uhne show in
\cite[Prop.~3.1]{cuntzkuehne:subgraphs} that
$\CA_G$ is  inductively free.
Thus in order to derive Conjecture \ref{conj:freeAG2} we need to show that $\CA_G$ is  inductively free for
$G$ an almost path or a path-with-triangle graph.

We observe that the arrangements $\CA_{A_{n,2}}$ are particular ideal arrangements in the root system of type $D_{n+1}$, cf.~\cite[Ex.~1.4]{cuntzkuehne:subgraphs}. The latter
are known to be inductively free, \cite[Thm.~1.7]{roehrle:ideal}.
Thus we know in particular that $\CA_{A_{n,2}}$ is inductively free for $n \ge 4$.
In the same vein, for $n = 5,6,7$ the arrangements $\CA_{A_{n,3}}$ are certain ideal 
arrangements in the root system of type $E_{n+1}$.
As all ideal 
arrangements are inductively free, owing to \cite[Thm.~1.4]{cuntzroehrleschauenburg:ideal},
so are  $\CA_{A_{n,3}}$, for $n = 5,6,7$.

\section{Inductive Factoredness: Proof of Theorem \ref{thm:factoredAG2}}
\label{s:thm:factoredAG2}

According to Theorem \ref{thm:factoredAG},
$\CA_G$ is factored if and only if $G$ is a path graph or
$G =\Delta_{n,1}$ for some $n \ge 2$.
For $G$ a path graph, $\CA_G$ is supersolvable, by \cite[Cor.~8.11]{cuntzkuehne:subgraphs} and supersolvable arrangements are inductively factored, cf.~\cite{jambuparis:factored} or \cite[Prop.~3.11]{hogeroehrle:factored}.

Thus, in order to derive Theorem \ref{thm:factoredAG2} we need to show that $\CA_G$ is inductively factored for
$G =\Delta_{n,1}$ with $n \ge 2$.
For that purpose we present an induction table
for an inductive factorization $\pi$ of $\Delta_{n,1}$ for each $n \ge 2$. We require the following auxiliary result. Let $P_n$ be path graph on $n$ nodes.
Here we use the following labelling of $\Delta_{n,1}$:

		
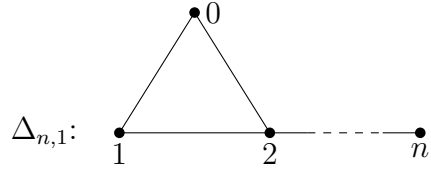
\begin{figure}[H]
			\begin{tikzpicture}
				\node (t) at (-1,0) {$\Delta_{n,1}$:};
				\node (v2h+) at (2.5,0) {};
				\node (v2kmid) at (3,0) {};
				\node (vkh-) at (3.5,0) {};
				\node (vk) at (0,0) {};
				\node (vn+1) at (1,1.6) {};
				\node (vk+1) at (2,0) {};
				\node (vk+1h+) at (2.5,0) {};
				\node (vk+1nmid) at (3,0) {};
				\node (vnh-) at (3.5,0) {};
				\node (vn) at (4,0) {};
				\fill ($(vk)$) circle[radius=2pt];
				\fill ($(vk+1)$) circle[radius=2pt];
				\fill ($(vn)$) circle[radius=2pt];
				\fill ($(vn+1)$) circle[radius=2pt];
				\node[below] at ($(vk)$) {$1$};
				\node[below] at ($(vk+1)$) {$2$};
				\node[below] at ($(vn)$) {$n$};
				\node[right] at ($(vn+1)$) {$0$};
    \draw ($(vk)$) -- ($(vk+1)$) -- ($(vk+1h+)$);
				\draw ($(vk)$) -- ($(vn+1)$) -- ($(vk+1)$);
				\draw[dashed] ($(vk+1h+)$) -- ($(vnh-)$);
				\draw ($(vnh-)$) -- ($(vn)$);
			\end{tikzpicture}
		\end{figure}

\begin{lemma}
	\label{lem:factoredAG2}
		For $n \ge 2$, consider the $n$-arrangement $\CB_n = \CA_{P_n} \bigcup \{ H_1,\ldots ,H_{n-1} \}$, where $H_k= \ker\left( 2x_0+x_1+\cdots+x_k \right)$ for $k=1,\ldots n-1$.
	Let $\pi=\pi^{n}$ be the partition of $\CB_n$ given by:
	\begin{align*}
		\pi_1  & :=  \{\ker(x_0)\} \\
		\pi_k & :=  \{\ker(x_j +\ldots+ x_{k-1}) \mid j = 0,\ldots ,k-1 \} \cup \{H_{k-1}\} \ \text{for}\ k=2, \ldots,n.
	\end{align*}
	Then $\pi$ is an inductive factorization of $\CB_n$.
	In particular, $\CB_n$ is inductively free with $\exp(\CB_n) = (1,3,4, \ldots, n+1)$.
\end{lemma}

\begin{proof}
	We aim to show that $H_1,\ldots,H_{n-1}$  forms an inductive chain (in that order) from the supersolvable arrangement $\CA_{P_n}$ to $\CB_n$. For that purpose define $\CA_0:=\CA_{P_n}$ and $\CA_i := \CA_{P_n} \cup \{H_1,\ldots, H_i\}$ for $i = 1, \ldots, n-1$. Both arrangements $\CB_2$ and $\CA_0$ are inductively factored. We proceed by induction on $n$ and $i$, assuming that $\CB_{n-1}$ and $\CA_{i-1}$ are inductively factored, with $\lambda=\lambda^{i-1}$ an inductive factorization for $\CA_{i-1}$ given by:
	\begin{align*}
		\lambda_k & :=  \pi_k \ \text{for}\ k=1, \ldots,i-1, \\
		\lambda_{k} &:= \{\ker(x_j +\ldots+ x_{k-1}) \mid j = 0,\ldots ,k-1 \}\ \text{for}\  k=i, \ldots,n.
	\end{align*}

	Note that after a change of coordinates for $\CA_{P_n}$, like the one in \cite[Ex.~1.2]{cuntzkuehne:subgraphs}, the partition $\lambda^0$ given by
	\begin{align*}
		\lambda_1^0 & :=  \pi_1, \\
		\lambda_{k}^0 &:=  \{\ker(x_j +\ldots+ x_{k-1}) \mid j = 0,\ldots ,k-1 \}\ \text{for}\ k = 2, \ldots, n,
	\end{align*}
	is the inductive factorization for the rank $n$ braid arrangement corresponding to the maximal chain of modular flats $$\{x_0=x_1\}<\{x_0=x_1=x_2\}<\cdots < \{x_0=x_1=\cdots =x_n\}.$$

	Then one checks that the restriction to $H_{i-1}$ yields the arrangement $\CB_{n-1}$ with the inductive partition $\pi^{n-1}$. Similarly, thanks to the chosen factorization, one sees by inspection that the map $\varrho\colon \CA_{i-1} \setminus \lambda_{i-1} \rightarrow \CA_{i-1}^{''}$ is bijective for each $i$.  Consequently,
	$\pi$ is an inductive factorization of $\CB_n$.
	We give the induction table for the inductive factorization in Table \ref{table0} below.
\end{proof}

	\begin{table}[ht!b]
		\renewcommand{\arraystretch}{1.5}
		\begin{tabular}{lllll}\hline
			$(\CA_i',\pi_i')$ &  $\exp\CA_i'$ & $\alpha_{H_i}$ & $(\CA_i'', \pi_i''$) & $\exp\CA_i''$\\
			\hline\hline
			$(\CA_0,\lambda^{0})$ &  $1,2,3,4, \ldots,n $ &  $2x_0+x_1$ & $(\CB_{n-1},\pi^{n-1})$ & $1,3,4,\ldots,n-1$ \\
			$(\CA_1,\lambda^{1})$ &  $1,3,3,4, \ldots,n $ &  $2x_0+x_1+x_2$ & $(\CB_{n-1},\pi^{n-1})$ & $1,3,4,\ldots,n-1$ \\
			\vdots & \vdots  & \vdots & \vdots  & \vdots  \\
			$(\CA_{n-2},\lambda^{n-2})$ &  $1,3,4,5, \ldots,n,n $ &  $2x_0+x_1+\cdots+x_{n-1}$ & $(\CB_{n-1},\pi^{n-1})$ & $1,3,4,\ldots,n-1$ \\
			$(\CB_n,\pi^{n})$ &  $1,3,4 \ldots,n+1$ & & &  \\

			\hline
		\end{tabular}
		\smallskip
		\caption{Induction Table of Factorizations for $\CB_n$}
		\label{table0}
	\end{table}

\begin{proof}[Proof of Theorem \ref{thm:factoredAG2}]
	Recall from \cite[Thm.~5.4]{cuntzkuehne:subgraphs} that $\CA_G$ is free  for $G =\Delta_{n,1}$  with
	$\exp(\CA_G) = (1,3,4, \ldots ,n+1, n+1)$.
	We claim that the following partition $\pi = \pi^{n+1}$ is an inductive factorization for $\CA_{\Delta_{n,1}}$:
	\begin{align*}
		\label{eq:piDelta}
		\pi_1 & :=  \{\ker(x_0 + x_1)\},\\
		\pi_k & :=  \{ \ker(x_i +\ldots+ x_k) \mid i = 0,\ldots ,k \} \ \text{for}\ k=2, \ldots,n;\  (k+1\  \text{hyperplanes}),\\
		\pi_{n+1} & :=  \{\ker(x_0 + x_2 + \ldots + x_k) \mid k = 2,\ldots,n \} \cup \{ \ker(x_0), \ker(x_1) \}\ (n+1\ \text{hyperplanes}).
	\end{align*}

	Note that this is identical to the partition in the proof of \cite[Prop.~8.9]{cuntzkuehne:subgraphs}. We argue by induction on $n \ge 2$. For $ n = 2$ one checks that $\pi$ is inductive. By induction we assume that $\pi$ given above is an inductive factorization for $\CA_{\Delta_{n-1,1}}$.

		We claim that  $H_1,\ldots, H_{n+2}$  is an inductive chain (in that order) from $\CA_{\Delta_{n-1,1}}$ to $\CA_{\Delta_{n,1}}$, where:
	\begin{align*}
		H_i &= \ker(x_{i-1}+\cdots +x_n)\ \text{for } i = 1,\ldots, n+1 \\
		H_{n+2} &= \ker(x_0+x_2+\cdots+x_n).
	\end{align*}

	For $i=1,\ldots, n+1$, let $\CA_i := \CA_{\Delta_{n-1,1}} \bigcup \{H_1,\ldots, H_i \}$ and $\lambda=\lambda^{i}$ be the corresponding partition obtained from the partition $\pi = \pi^{n-1}$ of $\CA_{\Delta_{n-1,1}}$ as follows:
	\begin{align*}
		\lambda_k & :=  \pi_k \ \text{for}\ 1\leq k\leq n-1 \\
		\lambda_{n} &:= \{H_1,\ldots, H_i\}\\
		\lambda_{n+1} &:= \pi_{n};
	\end{align*}
	the last hyperplane is added to the block $\lambda_{n+1}$, unlike the others, yielding the arrangement $\CA_{n+2}=\CA_{\Delta_{n,1}}$ and its inductive partition $\lambda^{n+2}=\pi^n$. \\

	In the first $n+1$ addition steps the restriction arrangement is always $\CA_{\Delta_{n-1,1}}$ with the respective inductive factorization $\pi^{n-1}$, and for example by using \cite[Lem.~2.10]{cuntzkuehne:subgraphs}, the map $\varrho$ can be seen to be always bijective.

	For the last hyperplane, the restriction is different: we leave the class of connected subgraph arrangements, but just for one hyperplane: $\CA_{\Delta_{n,1}}^{''} \cong \CA_{\Delta_{n-1,1}} \cup \{\ker(x_0-x_1)\}$, with the associated partition $(\pi^n)''$ obtained by adding the extra hyperplane $\ker(x_0-x_1)$ to the $(n-1)$-st (and penultimate) block of the inductive partition $\pi^{n-1}$ of $\CA_{\Delta_{n-1,1}}$. \\

	Now, we prove that this arrangement is inductively factored. By choosing as distinguished hyperplane $\ker(x_0-x_1)$, clearly, the deletion is  $\CA_{\Delta_{n-1,1}}$, and the restriction turns out to be the arrangement $\CB_{n-1}$ with the respective inductive factorization from Lemma \ref{lem:factoredAG2}. Hence, thanks to this lemma  we are done. See Table \ref{table1} for the induction table of this inductive factorization of $\CA_{\Delta_{n,1}}$.
\end{proof}

	\begin{table}[ht!b]
		\renewcommand{\arraystretch}{1.5}
		\begin{tabular}{lllll}\hline
			$(\CA_i',\pi_i')$ &  $\exp\CA_i'$ & $\alpha_{H_i}$ & $(\CA_i'', \pi_i''$) & $\exp\CA_i''$\\
			\hline\hline
			$(\CA_{\Delta_{n-1,1}},\pi)$ &  $1,3,4, \ldots, n, n$ &  $x_0+\cdots +x_n$ & $(\CA_{\Delta_{n-1,1}},\pi)$ & $1,3,4,\ldots, n,n$ \\
			$(\CA_1,\lambda^{1})$ &  $1,3,4, \ldots, n, 1,n$ &  $x_1+\cdots +x_n$ & $(\CA_{\Delta_{n-1,1}},\pi)$ & $1,3,4,\ldots, n,n$ \\
			\vdots & \vdots & \vdots & \vdots & \vdots \\
			$(\CA_n,\lambda^{n})$ &  $1,3,4, \ldots, n,n,n$ &  $x_n$ & $(\CA_{\Delta_{n-1,1}},\pi)$ & $1,3,4,\ldots, n,n$ \\
			$(\CA_{n+1},\lambda^{n+1})$ &  $1,3,4, \ldots, n, n+1,n$ &  $x_0+x_2+\cdots +x_n$ & $(\CA_{\Delta_{n,1}}^{''},\pi^{''})$ & $1,3,4,\ldots, n+1,n$ \\
			$(\CA_{\Delta_{n,1}},\pi)$ &  $1,3,4, \ldots, n+1, n+1$ & & &  \\

			\hline
		\end{tabular}
		\smallskip
		\caption{Induction Table of Factorizations for $\CA_{\Delta_{n,1}}$}
		\label{table1}
	\end{table}

\section{Aspherical $\CA_G$: Proof of Theorem \ref{thm:kpi1}}
\label{s:thm:kpi1}

\begin{remark}
	\label{rem:local}
	Thanks to an observation by Oka,
	asphericity is a local property in the sense of
	\S \ref{sub:local},
	e.g., see \cite[Lem.~1.1]{paris:deligne}.
\end{remark}

We recall the notion of a generic arrangement from \cite[Def.~5.22]{orlikterao:arrangements}.

\begin{remark}
	\label{rem:generic}
An 
$\ell$-arrangement $\CA$ with $\rank(\CA)=r$ is called \emph{generic} if 
every subarrangement $\CB$ of $\CA$ of cardinality $r$ is linearly independent and $|\CA| > r$, \cite[Def.~5.22]{orlikterao:arrangements}.
Due to work of Hattori, 
generic arrangements are not aspherical, cf.~\cite[Cor.~5.23]{orlikterao:arrangements}.
\end{remark}
Next we recall the ``simple triangle`` condition of Falk and Randell \cite[Cor.~3.3, (3.12)]{falkrandell:homotopy}.

\begin{defn}
	\label{def:simpleT}
	Let $\CA$ be an arrangement in $\BBR^3$. 
	Then $\CA$ is said to have a \emph{simple triangle} (cf.\ Figure \ref{fig:SimpT_G1}) if there is a chamber $C$, 
	i.e., a connected component of $\BBR^3 \setminus \cup_{H \in \CA}H$,
	such that:
	\begin{enumerate}[(i)]
		\item $|\CA|\geq 4$,
		\item $C$ is simplicial, i.e., a cone over a triangle, and
		\item if $\{H_1,H_2,H_3\}$ are the walls of $C$, then $\CA_{H_i\cap H_j} = \{H_i,H_j\}$ for $1\leq i<j\leq3$.
	\end{enumerate}
\end{defn}	

The notion of a simple triangle is linked to asphericity as follows.

\begin{lemma}[{\cite[Cor.~3.3]{falkrandell:homotopy}}]
	\label{lem:SimpleT}
	Let $\CA$ be an arrangement  in $\BBR^3$ admitting a simple triangle.
	Then (the complexified) arrangement $\CA$ is not $K(\pi,1)$.
\end{lemma}

The following is the list from \cite{cuntzkuehne:subgraphs}
of minimal graphs $G$ subject to $\CA_G$ being not free.

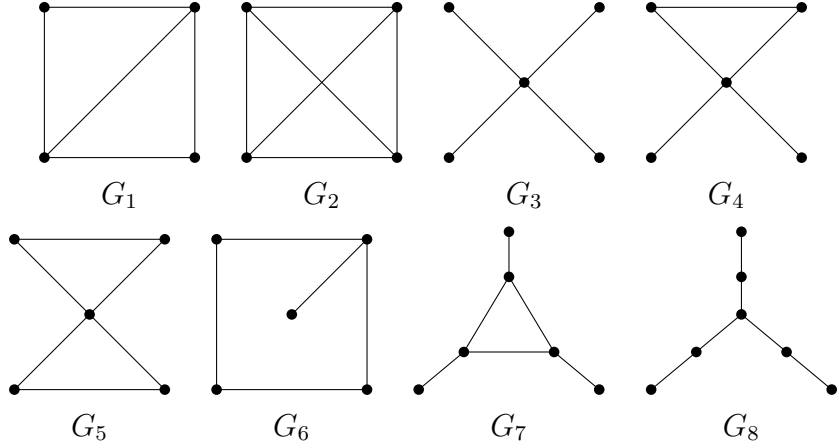
\begin{figure}[H]
	\begin{tikzpicture}
		\node (a1) at (2,2) {};
		\node (a2) at (0,2)  {};
		\node (a3) at (0,0)  {};
		\node (a4) at (2,0)  {};
		\foreach \from/\to in {a1/a2,a2/a3,a3/a4,a1/a4,a1/a3}
		\draw ($(\from)$) -- ($(\to)$);
		\fill ($(a1)$) circle[radius=2pt];
		\fill ($(a2)$) circle[radius=2pt];
		\fill ($(a3)$) circle[radius=2pt];
		\fill ($(a4)$) circle[radius=2pt];
		\node at (1,-0.5) {$G_1$};
	\end{tikzpicture}\quad 
	\begin{tikzpicture}
		\node (a1) at (2,2) {};
		\node (a2) at (0,2)  {};
		\node (a3) at (0,0)  {};
		\node (a4) at (2,0)  {};
		\foreach \from/\to in {a1/a2,a2/a3,a3/a4,a1/a4,a1/a3,a2/a4}
		\draw ($(\from)$) -- ($(\to)$);
		\fill ($(a1)$) circle[radius=2pt];
		\fill ($(a2)$) circle[radius=2pt];
		\fill ($(a3)$) circle[radius=2pt];
		\fill ($(a4)$) circle[radius=2pt];
		\node at (1,-0.5) {$G_2$};
	\end{tikzpicture}\quad
	\begin{tikzpicture}
		\node (a1) at (0,0) {};
		\node (a2) at (1,1)  {};
		\node (a3) at (-1,1)  {};
		\node (a4) at (-1,-1)  {};
		\node (a5) at (1,-1)  {};
		\foreach \from/\to in {a1/a2,a1/a3,a1/a4,a1/a5}
		\draw ($(\from)$) -- ($(\to)$);
		\fill ($(a1)$) circle[radius=2pt];
		\fill ($(a2)$) circle[radius=2pt];
		\fill ($(a3)$) circle[radius=2pt];
		\fill ($(a4)$) circle[radius=2pt];
		\fill ($(a5)$) circle[radius=2pt];
		\node at (0,-1.5) {$G_3$};
	\end{tikzpicture}\quad
	\begin{tikzpicture}
		\node (a1) at (0,0) {};
		\node (a2) at (1,1)  {};
		\node (a3) at (-1,1)  {};
		\node (a4) at (-1,-1)  {};
		\node (a5) at (1,-1)  {};
		\foreach \from/\to in {a1/a2,a1/a3,a1/a4,a1/a5,a2/a3}
		\draw ($(\from)$) -- ($(\to)$);
		\fill ($(a1)$) circle[radius=2pt];
		\fill ($(a2)$) circle[radius=2pt];
		\fill ($(a3)$) circle[radius=2pt];
		\fill ($(a4)$) circle[radius=2pt];
		\fill ($(a5)$) circle[radius=2pt];
		\node at (0,-1.5) {$G_4$};
	\end{tikzpicture}
	
	\begin{tikzpicture}
		\node (a1) at (0,0) {};
		\node (a2) at (1,1)  {};
		\node (a3) at (-1,1)  {};
		\node (a4) at (-1,-1)  {};
		\node (a5) at (1,-1)  {};
		\foreach \from/\to in {a1/a2,a1/a3,a1/a4,a1/a5,a2/a3,a4/a5}
		\draw ($(\from)$) -- ($(\to)$);
		\fill ($(a1)$) circle[radius=2pt];
		\fill ($(a2)$) circle[radius=2pt];
		\fill ($(a3)$) circle[radius=2pt];
		\fill ($(a4)$) circle[radius=2pt];
		\fill ($(a5)$) circle[radius=2pt];
		\node at (0,-1.5) {$G_5$};
	\end{tikzpicture}\quad
	\begin{tikzpicture}
		\node (a5) at (0,0) {};
		\node (a1) at (1,1)  {};
		\node (a2) at (-1,1)  {};
		\node (a3) at (-1,-1)  {};
		\node (a4) at (1,-1)  {};
		\foreach \from/\to in {a1/a2,a2/a3,a3/a4,a1/a4,a1/a5}
		\draw ($(\from)$) -- ($(\to)$);
		\fill ($(a1)$) circle[radius=2pt];
		\fill ($(a2)$) circle[radius=2pt];
		\fill ($(a3)$) circle[radius=2pt];
		\fill ($(a4)$) circle[radius=2pt];
		\fill ($(a5)$) circle[radius=2pt];
		\node at (0,-1.5) {$G_6$};
	\end{tikzpicture}\quad
	\begin{tikzpicture}
		\node (a1) at (0,1) {};
		\node (a2) at (-0.6,0) {};
		\node (a3) at (0.6,0) {};
		\node (a4) at (0,1.6) {};
		\node (a5) at (-1.2,-0.5) {};
		\node (a6) at (1.2,-0.5) {};
		\foreach \from/\to in {a1/a2,a1/a3,a2/a3,a1/a4,a2/a5,a3/a6}
		\draw ($(\from)$) -- ($(\to)$);
		\fill ($(a1)$) circle[radius=2pt];
		\fill ($(a2)$) circle[radius=2pt];
		\fill ($(a3)$) circle[radius=2pt];
		\fill ($(a4)$) circle[radius=2pt];
		\fill ($(a5)$) circle[radius=2pt];
		\fill ($(a6)$) circle[radius=2pt];
		\node at (0,-1) {$G_7$};
	\end{tikzpicture}\quad
	\begin{tikzpicture}
	\node (a1) at (0,0.5) {};
	\node (a2) at (0,1) {};
	\node (a3) at (-0.6,0) {};
	\node (a4) at (0.6,0) {};
	\node (a5) at (0,1.6) {};
	\node (a6) at (-1.2,-0.5) {};
	\node (a7) at (1.2,-0.5) {};
	\foreach \from/\to in {a1/a2,a1/a3,a1/a4,a2/a5,a3/a6,a4/a7}
	\draw ($(\from)$) -- ($(\to)$);
	\fill ($(a1)$) circle[radius=2pt];
	\fill ($(a2)$) circle[radius=2pt];
	\fill ($(a3)$) circle[radius=2pt];
	\fill ($(a4)$) circle[radius=2pt];
	\fill ($(a5)$) circle[radius=2pt];
	\fill ($(a6)$) circle[radius=2pt];
	\fill ($(a7)$) circle[radius=2pt];
	\node at (0,-1) {$G_8$};
	\end{tikzpicture}
	\caption{The graphs $G_1$ up to $G_8$}
	\label{fig:G1_8}
\end{figure}

We show next that $\CA_G$ is not $K(\pi, 1)$ for each $G$ in Figure \ref{fig:G1_8} with the possible exception of $G_2 = K_4$. We begin with $G_1$:

\begin{prop}
\label{prop:g1}
  The connected subgraph arrangement $\CA=\CA_{G_1}$ stemming from the graph
  
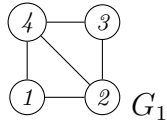
\begin{center}
  \begin{tikzpicture}
    [
    every node/.style={draw, circle, scale=1, inner sep=2pt},
    font=\scriptsize
    ]
    \node (1) at (0,0) {1};
    \node (2) at (1,0) {2} edge (1);
    \node (3) at (1,1) {3} edge (2);
    \node (4) at (0,1) {4} edge (1) edge (3) edge (2);
  \end{tikzpicture}
  	{$G_1$}
  \end{center}
  is not $K(\pi, 1)$.
\end{prop}
\begin{proof}
  Consider the rank $3$ localization $\CB'=\CA_X$ for the subspace
  $$X= H_{2}\cap H_{4}\cap H_{[4]} = \{x_2=x_4=x_1+
  x_3 = 0\}\,.$$
  This arrangement is defined by
  $$Q(\CB') = x_2x_4(x_2+x_4)(x_1+x_2+x_3)(x_1+x_3+x_4)(x_1+x_2+x_3+x_4)$$
  which is linearly isomorphic to the essential $3$-arrangement $\CB$ defined by
  $$Q(\CB) = yz(x+y)(x+z)(y+z)(x+y+z)$$
  under the linear transformation
  $\begin{pmatrix}x_1+x_3 \\ x_2 \\ x_4 \end{pmatrix}
  \mapsto
  \begin{pmatrix} x\\y\\z \end{pmatrix}$.
  For $t\in \BBC$, consider the arrangement $\CB(t)$ defined by
  $Q(\CB(t)) = yz(x+y+tz)(x+z)(y+z)(x+y+z)\,.$
  We have $\CB = \CB(0)$ and $\CB \cong \CB(t)$ for all $t\in \BBC \setminus \{1,2\}$.
  Furthermore, for $t$ real and $t>2$, the arrangement $\CB(t)$ has a ``simple triangle''.
  Defining the $1$-parameter family $\Big(\CB(2(e^{t\pi i}+1)) \mid 1\leq t \leq 2\Big)$, we observe that this is a \emph{lattice isotopy} in the sense of \cite{randell:lattice-isotopy}.
  Thus, the arrangements $\CB = \CB(0) = \CB(2(e^{\pi i}+1))$ and $\CB(4) = \CB(2(e^{2\pi i}+1))$ have diffeomorphic complements.
  Since $\CB(4)$ admits a simple triangle, and thus is not 
  $K(\pi, 1)$, owing to Lemma \ref{lem:SimpleT}, also 
  $\CB$ is not $K(\pi, 1)$.
  Since $\CB$ is linearly isomorphic to $\CA_X$, the latter is not $K(\pi, 1)$. Finally, it follows from Remark \ref{rem:local}
  that neither $\CA$ is $K(\pi, 1)$.
\end{proof}

\begin{remark}
	\label{rem:g7g8}
	It is shown in (the proof of) \cite[ Prop.~3.1]{mueckschroehrlewiesner} that for $G = G_7$ and $G_8$, $\CA_G$ admits a generic localization of rank $4$ and $5$, respectively. It follows from Remarks \ref{rem:local}  and \ref{rem:generic} that $\CA_G$ is not $K(\pi,1)$ either.
\end{remark}

\begin{figure}[h]
  \begin{tikzpicture}
    [scale=0.5]
    \draw (-4.5,0) -- (4.5,0);
    \draw (-1,-4) -- (-1,4);
    \draw (-4.5,-1) -- (4.5,-1);
    \draw (-3.5, 3.5) -- (3.5, -3.5);
    \draw (-3.5, 2.5) -- (3, -4);
    \draw (0,0) circle [radius=5.5];
  \end{tikzpicture}
  \quad
  \begin{tikzpicture}
    [scale=0.5]
    \draw (-4.5,0) -- (4.5,0);
    \draw (-1,-4) -- (-1,4);
    \draw (-4.5,-1) -- (4.5,-1);
    \draw (-5, 1) -- (0.5, -4.5);
    \draw (-3.5, 2.5) -- (3, -4);
    \draw (0,0) circle [radius=5.5];
    \draw[fill=gray!42] (-1,-1) -- (-1, -3) -- (-3, -1) -- (-1, -1);
  \end{tikzpicture}
        \caption{Projective pictures of the arrangements $\CB=\CB(0)$ and $\CB(4)$ for $z=1$ with the ``simple triangle'' shaded in.}
        \label{fig:SimpT_G1}
\end{figure}
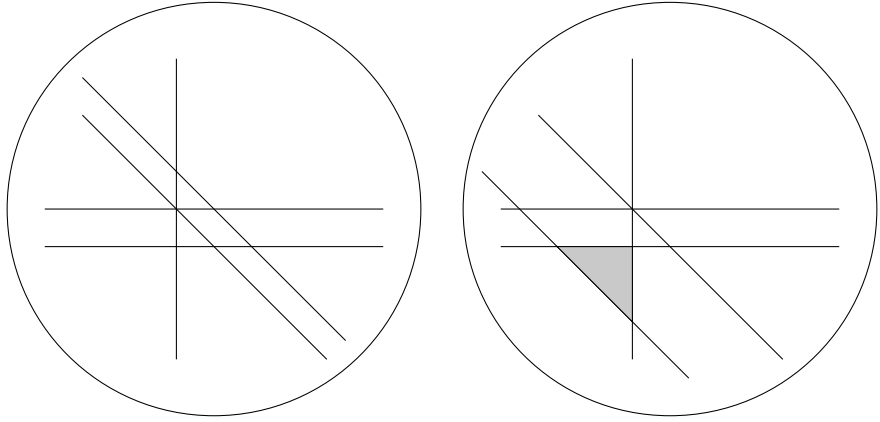

We note that the arrangement $\CB$ above is linearly isomorphic to the rank $3$ arrangement utilized in 
\cite[Lem.~2.1]{amendmoellerroehrle:aspherical}.

We discuss further obstructions to the $K(\pi,1)$ property for the $\CA_G$ and give 
additional minimal configurations of connected graphs $G$ for which $\CA_G$ fails to be aspherical in Figure \ref{fig:SC-nP}.
Note that these include all connected graphs $G$ with $5$ nodes which contain $K_4$ as a subgraph ($K_5$, $G_{16}$, $G_{17}$, $G_{18}$). These obstructions also stem from the presence of a
localization of rank $3$ which satisfies
the ``simple triangle'' condition from Definition \ref{def:simpleT}.

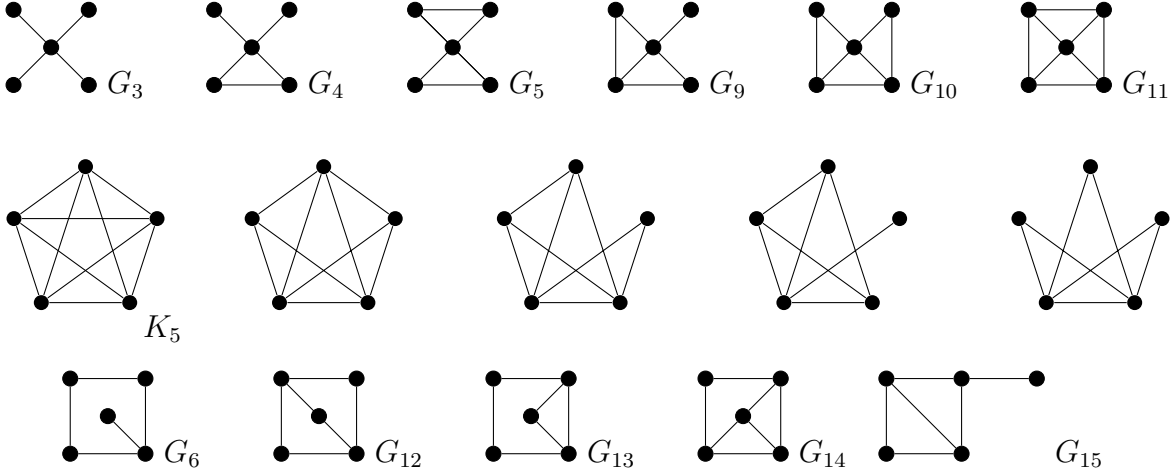
\begin{figure}[htbp]
	\centering
	\begin{tikzpicture}
		\draw (0,0) node[v](1){};
		\draw (1,0) node[v](2){};
		\draw (0.5,0.5) node[v](3){};
		\draw (0,1) node[v](4){};
		\draw (1,1) node[v](5){};
		\draw (4)--(2);
		\draw (5)--(1);
	\end{tikzpicture}
	$G_3$\qquad
	\centering
	\begin{tikzpicture}
		\draw (0,0) node[v](1){};
		\draw (1,0) node[v](2){};
		\draw (0.5,0.5) node[v](3){};
		\draw (0,1) node[v](4){};
		\draw (1,1) node[v](5){};
		\draw (1)--(2);
		\draw (4)--(2);
		\draw (5)--(1);
	\end{tikzpicture}
	$G_4$\qquad
	\begin{tikzpicture}
		\draw (0,0) node[v](1){};
		\draw (1,0) node[v](2){};
		\draw (0.5,0.5) node[v](3){};
		\draw (0,1) node[v](4){};
		\draw (1,1) node[v](5){};
		\draw (1)--(2)--(4)--(5);
		\draw (4)--(2);
		\draw (5)--(1);
	\end{tikzpicture}
	$G_5$\qquad
 \begin{tikzpicture}
		\draw (0,0) node[v](1){};
		\draw (1,0) node[v](2){};
		\draw (0.5,0.5) node[v](3){};
		\draw (0,1) node[v](4){};
		\draw (1,1) node[v](5){};
		\draw (1)--(2)--(4);
		\draw (5)--(1);
		\draw (4)--(1);
	\end{tikzpicture}
	$G_9$\qquad
	\begin{tikzpicture}
		\draw (0,0) node[v](1){};
		\draw (1,0) node[v](2){};
		\draw (0.5,0.5) node[v](3){};
		\draw (0,1) node[v](4){};
		\draw (1,1) node[v](5){};
		\draw (1)--(2);
		\draw (4)--(1);
		\draw (4)--(2);
		\draw (5)--(1);
		\draw (5)--(2);
	\end{tikzpicture}
	$G_{10}$\qquad
	\begin{tikzpicture}
		\draw (0,0) node[v](1){};
		\draw (1,0) node[v](2){};
		\draw (0.5,0.5) node[v](3){};
		\draw (0,1) node[v](4){};
		\draw (1,1) node[v](5){};
		\draw (1)--(2)--(5)--(4)--(1);
		\draw (4)--(2);
		\draw (5)--(1);
	\end{tikzpicture}
	$G_{11}$\\
	\bigskip
	\begin{tikzpicture}[bullet/.style={circle, fill, inner sep=2pt}]
		\foreach \lab [count=\c,
		evaluate=\c as \ang using {18+72*\c}]
		in {, , , , } {
			\node[bullet] (\c) at (\ang:10mm) {};
			\node at (\ang:14mm){$\lab$};
			\foreach \i in {1,...,\c} {
				\draw(\i)--(\c);
			}
		}
	\end{tikzpicture}\hglue-20pt $K_5$\quad
	\begin{tikzpicture}[bullet/.style={circle, fill, inner sep=2pt}]
		\foreach \lab [count=\c, evaluate=\c as \ang using {18+72*\c}] in {, , , , } {
			\node[bullet] (\c) at (\ang:10mm) {};
			\node at (\ang:14mm){$\lab$};
		}
		\draw(5)--(1)--(2)--(3)--(4)--(5)--(3)--(1)--(4)--(2);
	\end{tikzpicture}\hglue-20pt $G_{16}$\quad
	\begin{tikzpicture}[bullet/.style={circle, fill, inner sep=2pt}]
		\foreach \lab [count=\c, evaluate=\c as \ang using {18+72*\c}] in {, , , , } {
			\node[bullet] (\c) at (\ang:10mm) {};
			\node at (\ang:14mm){$\lab$};
		}
		\draw (1)--(2)--(3)--(4)--(5)--(3)--(1)--(4)--(2);
	\end{tikzpicture}\hglue-20pt $G_{17}$\quad
	\begin{tikzpicture}[bullet/.style={circle, fill, inner sep=2pt}]
		\foreach \lab [count=\c, evaluate=\c as \ang using {18+72*\c}] in {, , , , } {
			\node[bullet] (\c) at (\ang:10mm) {};
			\node at (\ang:14mm){$\lab$};
		}
		\draw (1)--(2)--(3)--(4);
		\draw (5)--(3)--(1)--(4)--(2);
	\end{tikzpicture}
	\hglue-20pt $G_{18}$\quad
	\begin{tikzpicture}[bullet/.style={circle, fill, inner sep=2pt}]
		\foreach \lab [count=\c, evaluate=\c as \ang using {18+72*\c}] in {, , , , } {
			\node[bullet] (\c) at (\ang:10mm) {};
			\node at (\ang:14mm){$\lab$};
		}
		\draw (2)--(3)--(4)--(5)--(3)--(1)--(4)--(2);
	\end{tikzpicture}
	\hglue-20pt $G_{19}$\\
	\bigskip
	\begin{tikzpicture}
		\draw (0,0) node[v](1){};
		\draw (1,0) node[v](2){};
		\draw (0.5,0.5) node[v](3){};
		\draw (0,1) node[v](4){};
		\draw (1,1) node[v](5){};
		\draw (1)--(2)--(5)--(4)--(1);
		\draw (3)--(2);
	\end{tikzpicture}
	$G_6$
\qquad
	\begin{tikzpicture}
		\draw (0,0) node[v](1){};
		\draw (1,0) node[v](2){};
		\draw (0.5,0.5) node[v](3){};
		\draw (0,1) node[v](4){};
		\draw (1,1) node[v](5){};
		\draw (1)--(2)--(5)--(4)--(1);
		\draw (4)--(2);
	\end{tikzpicture}
	$G_{12}$\qquad
	\begin{tikzpicture}
		\draw (0,0) node[v](1){};
		\draw (1,0) node[v](2){};
		\draw (0.5,0.5) node[v](3){};
		\draw (0,1) node[v](4){};
		\draw (1,1) node[v](5){};
		\draw (1)--(2)--(5)--(4)--(1);
		\draw (3)--(2);
		\draw (3)--(5);
	\end{tikzpicture}
	$G_{13}$\qquad
	\begin{tikzpicture}
		\draw (0,0) node[v](1){};
		\draw (1,0) node[v](2){};
		\draw (0.5,0.5) node[v](3){};
		\draw (0,1) node[v](4){};
		\draw (1,1) node[v](5){};
		\draw (1)--(2)--(5)--(4)--(1);
		\draw (3)--(2);
		\draw (3)--(5);
		\draw (3)--(1);
	\end{tikzpicture}
	$G_{14}$\quad
	\begin{tikzpicture}
		\draw (0,0) node[v](1){};
		\draw (1,0) node[v](2){};
		\draw (0,1) node[v](4){};
		\draw (1,1) node[v](5){};
		\draw (2,1) node[v](6){};
		\draw (1)--(2)--(5)--(4)--(1);
		\draw (4)--(2);
		\draw (5)--(6);
	\end{tikzpicture}
	$G_{15}$ \bigskip
		\caption{Minimal  connected graphs $G$ for which $\CA_G$ fails to be $K(\pi,1)$.}
		\label{fig:SC-nP}
	\end{figure}

\begin{prop}
	\label{prop:nonKpione}
	If $G$ is one of the graphs in Figure \ref{fig:SC-nP}, then $\CA_G$ is not $K(\pi,1)$.
\end{prop}

\begin{proof}
	In Table \ref{table:GsSimpT} we present a list of rank $3$ intersections $X \in L(\CA_G)$ for the graphs $G = G_3,\ldots,G_6$, 
	$K_5, G_{16},\ldots, G_{19}$ from Figure \ref{fig:SC-nP} such that $(\CA_G)_X$ has a simple triangle.
	To find these localizations, we used GAP4-code \cite{GAP4} developed by the third author, see \cite{GAP_hyparr_pkg}. 	Thus these rank $3$ localizations are not $K(\pi,1)$, by Lemma \ref{lem:SimpleT}, and so neither is $\CA_G$, by Remark \ref{rem:local}. In particular, they include the instances when $G$ is $G_3$,  $G_4$, $G_5$, or $G_6$ from Figure \ref{fig:G1_8}.

	\begin{table}
		\renewcommand{\arraystretch}{1.3}
		\begin{tabular}{lll}
			$G$ & $X$ & $(\CA_G)_X$ \\
			\hline \hline 
			\begin{tikzpicture}
				[scale=0.5,
				every node/.style={draw, circle, scale=0.75, inner sep=1pt},
				font=\tiny
				]
				\draw (0.5,0.5) node(1){$1$};
				\draw (0,0) node(4){$4$};
				\draw (1,0) node(5){$5$};
				\draw (0,1) node(2){$2$};
				\draw (1,1) node(3){$3$};
				\foreach \from/\to in {1/2,1/3,1/4,1/5}
				\draw (\from) -- (\to);
			\end{tikzpicture} $G_3$ & $H_{1}\cap H_{123}\cap H_{145}$ & $\{H_{1}, H_{123}, H_{145},H_{12345}\}$ \\
			\begin{tikzpicture}
				[scale=0.5,
				every node/.style={draw, circle, scale=0.75, inner sep=1pt},
				font=\tiny
				]
				\draw (0.5,0.5) node(1){$1$};
				\draw (0,0) node(4){$4$};
				\draw (1,0) node(5){$5$};
				\draw (0,1) node(2){$2$};
				\draw (1,1) node(3){$3$};
				\foreach \from/\to in {1/2,1/3,1/4,1/5,4/5}
				\draw (\from) -- (\to);
			\end{tikzpicture} $G_4$ & $H_{1}\cap H_{124}\cap H_{135}$ & $\{H_{1}, H_{124}, H_{135},H_{12345}\}$ \\
			\begin{tikzpicture}
				[scale=0.5,
				every node/.style={draw, circle, scale=0.75, inner sep=1pt},
				font=\tiny
				]
				\draw (0.5,0.5) node(1){$1$};
				\draw (0,0) node(4){$4$};
				\draw (1,0) node(5){$5$};
				\draw (0,1) node(2){$2$};
				\draw (1,1) node(3){$3$};
				\foreach \from/\to in {1/2,1/3,1/4,1/5,2/3,4/5}
				\draw (\from) -- (\to);
			\end{tikzpicture} $G_5$ & $H_{1}\cap H_{124}\cap H_{135}$ & $\{H_{1}, H_{124}, H_{135},H_{12345}\}$ \\
			\begin{tikzpicture}
				[scale=0.5,
				every node/.style={draw, circle, scale=0.75, inner sep=1pt},
				font=\tiny
				]
				\draw (0.5,0.5) node(1){$1$};
				\draw (0,0) node(5){$5$};
				\draw (1,0) node(4){$4$};
				\draw (0,1) node(2){$2$};
				\draw (1,1) node(3){$3$};
				\foreach \from/\to in {1/2,2/3,3/4,4/5,2/5}
				\draw (\from) -- (\to);
			\end{tikzpicture} $G_6$ & $H_{1234}\cap H_{23}\cap H_{25}$ & $\{H_{1234}, H_{23}, H_{25},H_{1245}\}$ \\
			
			\begin{tikzpicture}
				[scale=0.3,
				every node/.style={draw, circle, scale=0.75, inner sep=1pt},
				font=\tiny
				]				
				\draw (1.176,0.382) node(1){$1$};
				\draw (0.0,1.236) node(2){$2$};
				\draw (-1.176,0.382) node(3){$3$};
				\draw (-0.727,-1) node(4){$4$};
				\draw (0.727,-1)  node(5){$5$};
				\foreach \from/\to in {1/2,1/3,1/4,1/5,2/3,2/4,2/5,3/4,3/5,4/5}
				\draw (\from) -- (\to);
			\end{tikzpicture}$K_5$ & $H_{125}\cap H_{145}\cap H_{235}$ & $\{H_{125}, H_{145}, H_{235},H_{345}\}$ \\
			\begin{tikzpicture}
				[scale=0.3,
				every node/.style={draw, circle, scale=0.75, inner sep=1pt},
				font=\tiny
				]				
				\draw (1.176,0.382) node(1){$1$};
				\draw (0.0,1.236) node(2){$2$};
				\draw (-1.176,0.382) node(3){$3$};
				\draw (-0.727,-1) node(4){$4$};
				\draw (0.727,-1)  node(5){$5$};
				\foreach \from/\to in {1/2,1/4,1/5,2/3,2/4,2/5,3/4,3/5,4/5}
				\draw (\from) -- (\to);
			\end{tikzpicture}$G_{16}$ & $H_{1235}\cap H_{1345}\cap H_{25}$ & $\{H_{1235}, H_{1345}, H_{25},H_{45}\}$ \\
			\begin{tikzpicture}
				[scale=0.3,
				every node/.style={draw, circle, scale=0.75, inner sep=1pt},
				font=\tiny
				]				
				\draw (1.176,0.382) node(1){$1$};
				\draw (0.0,1.236) node(2){$2$};
				\draw (-1.176,0.382) node(3){$3$};
				\draw (-0.727,-1) node(4){$4$};
				\draw (0.727,-1)  node(5){$5$};
				\foreach \from/\to in {1/4,1/5,2/3,2/4,2/5,3/4,3/5,4/5}
				\draw (\from) -- (\to);
			\end{tikzpicture}$G_{17}$ & $H_{1235}\cap H_{1345}\cap H_{25}$ & $\{H_{1245}, H_{1345}, H_{25},H_{45}\}$ \\
			\begin{tikzpicture}
				[scale=0.3,
				every node/.style={draw, circle, scale=0.75, inner sep=1pt},
				font=\tiny
				]				
				\draw (1.176,0.382) node(1){$1$};
				\draw (0.0,1.236) node(2){$2$};
				\draw (-1.176,0.382) node(3){$3$};
				\draw (-0.727,-1) node(4){$4$};
				\draw (0.727,-1)  node(5){$5$};
				\foreach \from/\to in {1/4,2/3,2/4,2/5,3/4,3/5,4/5}
				\draw (\from) -- (\to);
			\end{tikzpicture}$G_{18}$ & $H_{1234}\cap H_{1345}\cap H_{24}$ & $\{H_{1234}, H_{1345}, H_{24},H_{45}\}$ \\
			\begin{tikzpicture}
				[scale=0.3,
				every node/.style={draw, circle, scale=0.75, inner sep=1pt},
				font=\tiny
				]				
				\draw (1.176,0.382) node(1){$1$};
				\draw (0.0,1.236) node(2){$2$};
				\draw (-1.176,0.382) node(3){$3$};
				\draw (-0.727,-1) node(4){$4$};
				\draw (0.727,-1)  node(5){$5$};
				\foreach \from/\to in {1/4,1/5,2/4,2/5,3/4,3/5,4/5}
				\draw (\from) -- (\to);
			\end{tikzpicture}$G_{19}$ & $H_{1235}\cap H_{1345}\cap H_{25}$ & $\{H_{1235}, H_{1345}, H_{25},H_{45}\}$\\
			\hline
		\end{tabular}\caption{Rank $3$ intersections $X \in L(\CA_G)$ such that $(\CA_G)_X$ admits a simple triangle.}\label{table:GsSimpT}
	\end{table}

	Finally, observe that $G_1$ can be obtained from graph contractions from 
	$G = G_9$ up to $G_{15}$, so that 
	$\CA_{G_1}$ is a localization of $\CA_G$. It follows from 
	Lemma \ref{lem:local}, 
	Remark \ref{rem:local}, and Proposition \ref{prop:g1} that $\CA_G$ is not $K(\pi,1)$ in these instances as well.
\end{proof}

Armed with Proposition \ref{prop:nonKpione}, 
we obtain some immediate consequences.

\begin{corollary}
	\label{cor:nonKpioneK4}
	Let $G$ be a connected graph on $n \ge 5$ vertices. If $G$ admits $K_4$ as a subgraph, then $\CA_G$ is not $K(\pi,1)$. 
\end{corollary}

\begin{proof}
	After possibly applying some edge contractions to $G$, we arrive at a 
	connected graph with $5$ nodes which contains $K_4$ as a subgraph. As Figure \ref{fig:SC-nP} contains all such,
	it follows from Proposition \ref{prop:nonKpione} along with Lemma \ref{lem:local} and Remark \ref{rem:local} that $\CA_G$ is not $K(\pi,1)$.
\end{proof}

We record the special case of resonance arrangements from the corollary.

	\begin{corollary}
		\label{cor:nonKpioneCompleteGraph}
		Let $G = K_n$ be the complete graph on $n \ge 5$ vertices. Then $\CA_G$ is not $K(\pi,1)$. I.e., resonance arrangements of rank at least $5$ are not $K(\pi,1)$.
	\end{corollary}
	
Other families of connected subgraph arrangements can also be considered. He is such an instance. 
	
	\begin{corollary}
		\label{cor:nonKpioneG1G5}
		Let $G$ be a connected graph on $n \ge 5$ vertices. Suppose $G$ contains at least two cycles. Then $\CA_G$ is not $K(\pi,1)$. 
	\end{corollary}
	
	\begin{proof}
			After possibly applying some edge contractions to $G$, we arrive at a 
		connected graph with $5$ nodes which either contains $K_4$ as a subgraph or coincides with $G_{9}$ or $G_{15}$ from Figure \ref{fig:SC-nP}. Again, 
		it follows from Proposition \ref{prop:nonKpione} along with Lemma \ref{lem:local} and Remark \ref{rem:local} that $\CA_G$ is not $K(\pi,1)$.
	\end{proof}

Finally, we address Theorem \ref{thm:kpi1}.

\begin{proof}[Proof of Theorem \ref{thm:kpi1}]
	If $\CA_G$ has rank $3$, then either $G = P_3$ or $C_3$. In both instances $\CA_G$ is $K(\pi,1)$, thanks to \cite[Thm.~7.2]{cuntzkuehne:subgraphs}, and free by Theorem \ref{thm:freeAG}. If $\CA_G$ has rank $4$, then $G = P_4$, $C_4$, $A_{3,2}$ $\Delta_{3,1}$, $G_1$ or $G_2 = K_4$. 
	The result now follows from Theorem \ref{thm:freeAG} and Proposition \ref{prop:g1}.

	If $\CA_G$ has rank at least $5$, our argument closely follows  the line of reasoning of \cite[Thm.~6.7]{cuntzkuehne:subgraphs}:
For a given $G$ with at least $5$ nodes we use the contraction method from \S \ref{sub:local} and arrive at a suitable induced subgraph $G[S]$ of rank $5$ which
either belongs to one of our accepted families  or it appears in Figure \ref{fig:SC-nP}. In the latter case, $\CA_{G[S]}$ is not $K(\pi,1)$, thanks to Proposition \ref{prop:nonKpione}. Consequently,
owing to Lemma \ref{lem:local} and Remark \ref{rem:local} also $\CA_G$ fails to be $K(\pi,1)$.
We use this repeatedly in the arguments below.

	(1). Suppose there is a node $v$ in $G$ of degree at least $4$. Let $v_1, \ldots, v_4$ be four distinct neighbours of $v$. Let $S = \{v, v_1, \ldots, v_4\}$. Then the induced subgraph $G[S]$  of $G$ by $S$ is one of the graphs listed in the first two rows of Figure \ref{fig:SC-nP}.
	It follows from the argument outlined above that
	$\CA_G$ can't be $K(\pi,1)$.
	 Consequently, any node in $G$ has degree at most $3$.

	(2). Assume that $G$ has at least two cycles $C_1$ and $C_2$ of length at least three that share at least one
	edge. Contracting edges results in an induced subgraph $G[S] = G_1$, since any node in $G$ has degree at most $3$. It follows from 
 Lemma \ref{lem:local} and Proposition \ref{prop:g1} that $\CA_G$ fails to be $K(\pi,1)$.
	So there is no pair of such cycles in $G$.

	(3). Now suppose $G$ admits at least two cycles of length at least $3$ which do not share an edge. Then contracting edges in $G$ to a minimum of $5$ nodes leads to an induced subgraph $G[S]$ identical to $G_5$ in
	Figure \ref{fig:SC-nP}.
	Once again we conclude that $\CA_G$ is not $K(\pi,1)$.
	So there is no pair of such cycles in $G$ either.
	So there can be at most one cycle in $G$.

	(4). Next we suppose that $G$ admits a single cycle of length at least $4$, but that $G$ itself is not a cycle itself. Since the rank of $\CA_G$ is assumed to be at least $5$, contracting edges in $G$ leads to an induced subgraph $G[S]$ which is identical to $G_6$ in
	Figure \ref{fig:SC-nP}, as $G$ itself is not a cycle. Once again, we conclude that $\CA_G$ must fail to be $K(\pi,1)$.
	So there is at most one cycle in $G$ of length $3$.

	(5). Now suppose that $G$ does admit a cycle of length $3$. Then at least one of its nodes  is of degree $3$ (else $G = C_3$). Suppose there is yet another node in $G$ of degree $3$ which does not belong to the 3-cycle. After possibly contracting edges in $G$ we arrive at an induced subgraph $G[S]$ which is identical to $G_4$ in
	Figure \ref{fig:SC-nP}.
	Yet again, $\CA_G$ is not $K(\pi,1)$. 
 Suppose $G$ admits three nodes of degree three which belong to the cycle of length $3$. The induced subgraph of this configuration coincides with the graph $G_7$ from Figure \ref{fig:G1_8}. It follows from Lemma \ref{lem:local} and Remark \ref{rem:g7g8} that 
 $\CA_G$ fails to be $K(\pi,1)$. 
 So $G$ is $\Delta_{n,k}$, for some $n\ge4$.

	(6). Finally, we consider the case when $G$ is a tree and not itself a path graph. If $G$ admits more than one node of degree $3$, then contracting edges leads to an induced subgraph $G[S]$ which is identical to $G_3$ in
	Figure \ref{fig:SC-nP}. So there is at most one node $v$ of degree $3$. If the paths attached to the node $v$ all have lengths at least $2$, then there is an induced subgraph of $G$ identical to $G_8$ from Figure \ref{fig:G1_8}. It follows from Lemma \ref{lem:local} and Remark \ref{rem:g7g8} that 
 $\CA_G$ fails to be $K(\pi,1)$. 
 So therefore, $G$ is $A_{n,k}$, for some $n\ge4$.
\end{proof}

In view of Theorem \ref{thm:kpi1} and its proof above, 
it would be interesting to know whether $\CA_G$ is $K(\pi,1)$ when $G = C_4$ or $G = K_4$.

\section{Combinatorial Formality: Proof of Theorem \ref{thm:AGformal}}
\label{s:thm:AGformal}

A property for arrangements is said to be \emph{combinatorial} if it only depends on the intersection lattice of the underlying arrangement.
Yuzvinsky \cite[Ex.~2.2]{yuzvinsky:obstruction} demonstrated that formality is not combinatorial, answering a question raised by Falk and Randell \cite{falkrandell:homotopy} in the negative.
Yuzvinsky's insight motivates the following notion from \cite{moellermueckschroehre:formal}.

\begin{defn}
	\label{def:combformal}
	Suppose $\CA$ is a formal arrangement. We say $\CA$ is \emph{combinatorially formal} if every arrangement with 
	an intersection lattice isomorphic to the one of
	$\CA$ is also formal.
\end{defn}

The following definitions, which are originally due to Falk  for matroids \cite{falk:line-closure}, were adapted for arrangements in \cite[\S 2.4]{moellermueckschroehre:formal}.
Let $\CB \subset \CA$ be a subset of hyperplanes.
We say $\CB$ is \emph{closed} if $\CB =\CA_Y$ for $Y=\bigcap\limits_{H\in \CB} H$.
We call $\CB $ \emph{line-closed}
if for every pair $H,H'\in \CB $ of hyperplanes, we have $\CA_{H\cap H'}\subset \CB $.
The \emph{line-closure} $\lc(\CB)$ of $\CB $ is defined
as the intersection of all line-closed subsets of $\CA$ containing $\CB$.
The arrangement $\CA$ is called \emph{line-closed}
if every line-closed subset of $\CA$ is closed.
With these notions, we have the following criterion for combinatorial formality, see \cite[Cor.~3.8]{falk:line-closure}, \cite[Prop.~3.2]{moellermueckschroehre:formal}:

\begin{prop}
	\label{prop:lcbasis}
	Let $\CA$ be an arrangement of rank $r$. Suppose $\CB \subseteq \CA$ consists of $r$ hyperplanes such that $r(\CB)=r$ and $\lc(\CB)=\CA$. Then $\CA$ is combinatorially formal.
\end{prop}

A subset $\CB \subseteq \CA$ as in Proposition \ref{prop:lcbasis} is called an \emph{lc-basis} of $\CA$.

\begin{proof}[{Proof of Theorem \ref{thm:AGformal}}]
	Let $G$ be a connected graph.
	Let $\CB = \{\ker x_i \mid i \in [n]\} \subseteq \CA_G$.
	Then it is easy to see that successively all $H_I$ for $I \subseteq [n]$ with $G[I]$ connected belong to the line-closure $\lc(\CB)$ of $\CB$. Consequently, $\lc(\CB) = \CA_G$. Since
	$\rk(\CA_G) = n = |\CB| = \rk(\CB)$, it follows from Proposition \ref{prop:lcbasis}
	that $\CA_G$ is combinatorially formal.
\end{proof}

There is a stronger notion of formality for an arrangement $\CA$, that of \emph{$k$-formality} for $1 \le k \le \rk(\CA)$ due to Brandt and Terao \cite{brandtterao}.
In view of Theorem \ref{thm:AGformal} and in view of the fact that all free arrangements are not just formal but are $k$-formal for all $k$, one might ask for this stronger notion of $k$-formality among connected subgraph arrangements.

It was shown in \cite[Prop.~3.4]{mueckschroehrlewiesner} that despite the fact that both simple arrangements $\CA_{G_1}$ and $\CA_{G_2}$ are not free, both admit free multiplicities. It thus  follows from 
\cite[Cor.~4.10]{DiPasquale} that both $\CA_{G_1}$ and $\CA_{G_2}$ are $k$-formal for any $k$.
Thus $\CA_G$ is $k$-formal for any $k$ provided the rank of $\CA_G$ is at most $4$ or $\CA_G$ is free. Computational evidence for further non-free small rank connected subgraph arrangements suggests the following.

\begin{conjecture}
	Any $\CA_G$ is $k$-formal for any $k$.
\end{conjecture}

\section{On ideal subarrangements of $\CA_G$}
\label{sec:idealAG}

In this section we investigate properties of natural subarrangements 
of connected subgraph arrangements which are counterparts of ideal subarrangements of Weyl arrangements.

\subsection{Arrangements of ideal type in Weyl arrangements}
\label{sec:ideal}
We begin by recalling the construction of ideal arrangements from \cite[\S 11]{sommerstymoczko}.
Let $\Phi$ be an irreducible, reduced root system
and let $\Phi^+$ be the set of positive roots
with respect to some set of simple roots $\Pi$.
An \emph{(upper) order ideal},
or simply \emph{ideal} for short,
of  $\Phi^+$, is a subset $\CI$  of $\Phi^+$
satisfying the following condition:
if $\alpha \in \CI$ and $\beta \in \Phi^+$ so that
$\alpha + \beta \in \Phi^+$, then $\alpha + \beta \in \CI$.
Recall the standard partial ordering
$\preceq$ on $\Phi$: $\alpha \preceq \beta$
provided $\beta - \alpha$ is a $\BBZ_{\ge0}$-linear combination
of positive roots, or $\beta = \alpha$. Then $\CI$ is an ideal in $\Phi^+$
if and only if whenever
$\alpha \in \CI$ and $\beta \in \Phi^+$ so that
$\alpha \preceq \beta$, then $\beta \in \CI$.

Let $\CA(\Phi)$ be the \emph{Weyl arrangement} of $\Phi$,
i.e., $\CA(\Phi) = \{ H_\alpha \mid \alpha \in \Phi^+\}$,
where $H_\alpha$ is the hyperplane in the Euclidean space
$V = \BBR \otimes \BBZ \Phi$ orthogonal to the root $\alpha$.
Following \cite[\S 11]{sommerstymoczko},
we associate with an ideal $\CI$ in $\Phi^+$ the arrangement
consisting of all hyperplanes with respect to the roots in $\CI^c := \Phi^+ \setminus \CI$, the complement of $\CI$ in
$\Phi^+$.

\begin{defn}[{\cite[\S 11]{sommerstymoczko}}]
	\label{def:idealtype}
 Let $\CI \subseteq \Phi^+$ be an ideal and let
$\CI^c := \Phi^+ \setminus \CI$
be its complement in
$\Phi^+$.
	The \emph{arrangement of ideal type} associated with
	$\CI$ is the subarrangement $\CA_\CI$
	of $\CA(\Phi)$ defined by
	\[
	\CA_\CI := \{ H_\alpha \mid \alpha \in \CI^c\}.
	\]
 We note that $\CI^c$
 is a lower order ideal in the set of positive roots  $\Phi^+$.  
\end{defn}

In \cite{roehrle:ideal} it is shown that a combinatorial property (see 
\cite[Cond.\ 1.10]{roehrle:ideal})
combined with Terao's fibration theorem \cite{terao:modular}
provides an inductive method which allows one to deduce that a large class of
the arrangements of ideal type $\CA_\CI$
are inductively free.
In fact all ideal arrangements
$\CA_\CI$ are inductively free for all types of reduced root systems, thanks to 
\cite[Thm.~1.4]{cuntzroehrleschauenburg:ideal}.

\subsection{Arrangements of ideal type in connected subgraph arrangements}
\label{sec:ideal}

In \cite[Ex.~1.4]{cuntzkuehne:subgraphs}, Cuntz and K\"uhne remark that if $G$ is the underlying simple graph of the Dynkin diagram of a reduced root system $\Phi$, then the corresponding connected subgraph arrangement $\CA_G$ is a particular ideal arrangement $\CA_\CI$ as in Definition \ref{def:idealtype} where the hyperplanes in $\CA_G$ correspond to the positive roots $\alpha$ in $\Phi$ 
where a simple root has coefficient $0$ or $1$ in~$\alpha$.

Here Cuntz and K\"uhne also indicate that if the graph $G$ is not the Dynkin diagram of a reduced root system so that the Coxeter group with underlying Dynkin diagram $G$ is infinite, one may  
interpret $\CA_G$ as a finite ideal subarrangement 
of the infinite root system associated with $G$.

Generalizing this interpretation, 
for a simple graph $G$ we consider the poset of induced connected subgraphs of $G$ ordered by inclusion. A lower order ideal $\CI$ in this poset then gives rise to a subarrangement $\CA_\CI$ of 
$\CA_G$ (generalizing ideal subarrangements of Weyl arrangements discussed above) to which we also refer as an 
\emph{ideal subarrangement}  of $\CA_G$; i.e., for such an ideal $\CI$, define 
\[
\CA_\CI := \{H_I \in \CA_G \mid I \in \CI \}.
\]

It was shown in \cite[Thm.~1.1]{abeetall:weyl} that all ideal subarrangements $\CA_\CI$ of Weyl arrangements are free by means of Theorem \ref{Thm_MAT}, consequently the $\CA_\CI$ are MAT-free, see Definition \ref{def:MATfree}. 
Generalizing this result and generalizing \cite[Thm.~4.1]{cuntzkuehne:subgraphs}, 
we observe next that MAT-freeness also prevails for the class of ideal subarrangements $\CA_\CI$ of the connected subgraph arrangements $\CA_G$ for $G$ an almost path graph $A_{n,k}$.

\begin{prop}
\label{prop:idealAG}    
Let $G$ be  an almost path graph $A_{n,k}$ and let 
$\mathcal{I}$ be an ideal in the poset of induced subgraphs of $A_{n,k}$, and let $\CA_\CI$ be the associated ideal subarrangement of $\CA_G$.
Then $\CA_\CI$ is MAT-free.
\end{prop}

\begin{proof}
    The case $\CI = G$ 
    is \cite[Thm.~4.1]{cuntzkuehne:subgraphs}.
	The argument in the proof of \cite[Thm.~4.1]{cuntzkuehne:subgraphs} generalizes to arrangements stemming from ideals of such posets for non-Dynkin type almost-path graphs $A_{n,k}\neq E_6,E_7,E_8$.
	 The MAT-partition needed for $\CA_\CI$  is the same as in the proof of \cite[Thm.~4.1]{cuntzkuehne:subgraphs}, simply restricted to the ideal $\mathcal{I}$: let $d$ be the maximal rank of $\mathcal{I}$; define: \begin{align*} &\CA_0 := \Phi_d \text{ the empty arrangement in } \mathbb{Q}^{d} \\ &\CA_{i} := \CA_{i-1} \cup \{H_I\mid H_I\in \CA_{\mathcal{I}} \text{ and } |I|=i\} \text{ for }1\leq i\leq d. \end{align*}
	The proof of \cite[Thm.~4.1]{cuntzkuehne:subgraphs} can then be repeated almost verbatim: all hypotheses of Theorem \ref{Thm_MAT} 
 are satisfied. The arrangement $\CA_1$ is Boolean, so there is nothing to be shown for the first step $(\CA_0,\CA_1)$. For the pair $(\CA_{i-1},\CA_i)$ we observe that:
	\begin{enumerate}
		\item the hyperplanes in $\CA_i\setminus \CA_{i-1}$ are linearly independent, as the argument in the proof of \cite[Thm.~4.1]{cuntzkuehne:subgraphs} shows that all rank $i$ elements of the poset of induced subgraphs are independent; 
		\item the first part entails $\bigcup_{H \in \CA_i\setminus \CA_{i-1}} H \nsubseteq \bigcap_{H' \in \CA_{i-1}} H'$; 
		\item the condition $|\CA_{i-1}|-|(\CA_{i-1}\cup \{H_I\})^{H_I}| = i$ for $H_I\in \CA_i\setminus \CA_{i-1}$ is proven for $\CA_{\mathcal{I}} = \CA_{A_{n,k}}$, and it clearly only uses elements below $H_I$ in the poset ordering. Because $\mathcal{I}$ is an order ideal, the argument can be repeated in the present case.
	\end{enumerate}  
 This completes the proof of the proposition.
\end{proof}

For $\CA_G$ a connected subgraph arrangement and $s \in \BBN$, we define the ideal subarrangement
 $\CA_G^s$ of $\CA_G$ by the hyperplanes $H_I$ where $I$ involves at most $s$ coordinate functions, i.e.,  
\begin{equation}
    \label{eq:AGs}
    \CA_G^s := \{H_I \in \CA_G \mid |{I}| \le s \}.
\end{equation}
E.g.~$\CA_G^1$ is the Boolean subarrangement of $\CA_G$ consisting of the coordinate hyperplanes. 
So $\CA_G^1$ is supersolvable and $K(\pi,1)$ irrespective of $G$. We are going to investigate combinatorial and topological properties of the ideal subarrangements $\CA_G^s$ of $\CA_G$ for $s > 1$ for some prominent classes of graphs $G$ below.

\begin{remark}
    \label{rem:Pn}
    Let $G = P_n$ be a path graph. Then any ideal subarrangement of $\CA_G$ is an ideal subarrangement $\CA_\CI$ of the Weyl arrangement of type $A_n$. It is known that each such is supersolvable and thus is also $K(\pi,1)$, see \cite[\S 6]{hultman:koszul}, or 
    \cite[Thm.~1.5]{roehrle:ideal}. This in particular applies to $\CA_G^s$ for any $s \ge 1$.
\end{remark}

\begin{prop}
    \label{prop:AG2notkpi1}
    Suppose $G$ contains $K_3$ as a subgraph. Then $\CA_G^2$ is neither free, nor $K(\pi,1)$. 
    In particular, this is the case if $G$ is a complete graph $K_n$ or a path-with-triangle graph $\Delta_{n,k}$. 
\end{prop}

\begin{proof}
Suppose first $\CA_G$ has rank $3$. Then $G = K_3 = C_3 = \Delta_{2,1}$ and $\CA_G^2$ coincides with the arrangement $X_3$ from \cite[(3.12)]{falkrandell:homotopy}. Owing to \emph{loc.~cit.}, $\CA_G^2$ is neither free, nor $K(\pi,1)$. Now suppose that the rank of $\CA_G$ is at least $4$. 
Since $G$ admits $K_3$ as a subgraph, let $S \subset N$ be the set of nodes in $G$ defining a copy of a $K_3$ subgraph of $G$. The argument of the proof of Lemma \ref{lem:local}(i) in \cite[Lem.~6.2]{cuntzkuehne:subgraphs} applies to the ideal subarrangement $\CA_G^2$ of $\CA_G$ so that $\CA_{G[S]}^2\cong X_3$ is a localization of $\CA_G^2$.
It follows from Lemma \ref{lem:local}(i), the rank $3$ case above, \cite[Thm.~4.37]{orlikterao:arrangements} and Remark \ref{rem:local} that $\CA_G^2$ is neither free, nor $K(\pi,1)$. 
\end{proof}

\begin{remark}
    \label{rem:AGs}
Let $\beta$ be in $\Phi^+$. Then $\beta = \sum_{\alpha \in \Pi} c_\alpha \alpha$
for $c_\alpha \in \BBZ_{\ge0}$.
The \emph{height} of $\beta$ is defined to be $\hgt(\beta) = \sum_{\alpha \in \Pi} c_\alpha$.
Let $\theta$ be the highest root in $\Phi$.
Then $h = \hgt(\theta) + 1$ is the 
\emph{Coxeter number} of $W$. For $1 \le t \le h$, let 
$\CI_t :=\{\alpha \in \Phi^+ \mid \hgt(\alpha) \ge t\}$ be the ideal consisting of all roots of height at least $t$.
In particular, we have $\CI_1 = \Phi^+$ and $\CI_h = \varnothing$.
    
It is shown in \cite[Thm.~1.31]{roehrle:ideal} 
that if 
$\Phi$ is of type $D_n$ for $n \ge 4$, $E_6$, $E_7$, or $E_8$, 
and $\CI \supseteq \CI_3$ (resp.~$\CI \supseteq \CI_4$), 
then $\CA_\CI$ is supersolvable (resp.~inductively factored).
In particular, if $\CI = \CI_3$ (resp.~$\CI = \CI_4$), then $\CA_\CI$ consists of all hyperplanes in  $\CA(\Phi)$ relative to roots of height at most $2$ (resp.~$3$). Using \eqref{eq:AGs}, it follows that in these cases $\CA_{\CI_3} = \CA_G^2$ is supersolvable (resp.~$\CA_{\CI_4} = \CA_G^3$ is inductively factored), where  
$G$ is the underlying simple graph of the Dynkin diagram of $\Phi$. 
Next we show that \cite[Thm.~1.31]{roehrle:ideal} generalizes to all almost-path graphs $A_{n,k}$. \end{remark}

In our following two results we present 
particular ideal subarrangements of $\CA_G$ for 
$G = A_{n,k}$ which are inductively factored.

\begin{prop}
    \label{prop:AGs} 
    Let $G = A_{n,k}$ be the almost-path graph, for $n \ge 3$ and $k \ge 2$. Then 
$\CA_G^2$ is supersolvable, and 
$\CA_G^3$ is inductively factored.
\end{prop}

\begin{proof}
The inductive arguments in 
\cite{roehrle:ideal} to derive \cite[Thm.~1.31]{roehrle:ideal} equally apply to the 
generalizations for all almost-path graphs
$A_{n,k}$, for both parts.
Thanks to \cite[Thm.~1.31]{roehrle:ideal}, for $k = 2$ and any $n$ and for $k = 3$ and $n \le 7$, both parts above hold. 
So we may assume that $n \ge 7$ and $k \ge 4$ and argue by induction on $n$. 
Consider the subgraph $G' = A_{n-1,k}$ of $G$ and the induced 
subarrangement $\CA_{G'}$  of $\CA_G$. 
Let $X$ be the flat in $L(\CA_G^2)$ (resp.~in $L(\CA_G^3)$) given by the intersection of all $H_I$, where $I \subseteq [n-1]$. Then $X$ is a modular element of corank $1$ in $\CA_G^2$
(resp.~$\CA_G^3$) and 
$\CA_{G'}^2$ (resp.~$\CA_{G'}^3$) coincides with this localization of $\CA_G^2$ at $X$ (resp.~$\CA_G^3$ at $X$). It follows from  \cite[Thm.~1.12, Lem.~3.1, Lem.~3.4]{roehrle:ideal} that $\CA_G^2$ is supersolvable (resp.~$\CA_G^3$ is inductively factored), since  by induction, $\CA_{G'}^2$ is supersolvable (resp.~$\CA_{G'}^3$ is inductively factored).
\end{proof}

Next we consider another family of ideal subarrangements $\CA_\CI$ of the $\CA_G$  for $G = A_{n,k}$ each of its members lies properly in $\CA_{A_{n,k}}^4$ and properly contains $\CA_{A_{n,k}}^3$.

\begin{prop}
    \label{prop:AGs2}
    Let $G =  A_{n,k}$ for $n\ge 4$. Let $\CI$ be a lower order ideal in the poset of connected subgraphs of $G$ consisting of subsets $I \subset [n+1]$ subject to 
   $\CA_{A_{n,k}}^3 \subsetneq \CA_\CI  \subsetneq \CA_{A_{n,k}}^4$ 
    and $I \ne \{k-1, k, k+1, n+1\}$.  Then
$\CA_\CI$ is inductively factored.     
\end{prop}

\begin{proof}
    We argue by induction on $n$. Let $n = 4$. So $\CA_\CI$ is an ideal subarrangement of the Weyl arrangement of type $D_5$ with either $14$ or $15$ hyperplanes and exponents $\exp \CA_\CI = \{1,3,3,3,4\}$ or $\{1,3,3,4,4\}$.  
    One checks directly that in each instance $\CA_\CI$ is inductively factored. 
    For instance, if $\CA_\CI$ admits $15$ hyperplanes, then 
    the following is an inductive factorization of 
    $\CA_\CI$, where we use the short hand notation $abc$ for $\ker (x_a +x_b+ x_c)$:
    \[
\pi = \{2\}, \{1,12,125\}, \{3,23,123\}, \{5,25,235,2345\}, \{4,34,234,1234\}. 
    \]
    In fact there are four inductive factorizations of $\CA_\CI$, but all are equivalent under the automorphism group of $\CA_\CI$.
    Now suppose $n > 4$ and that the statement in the lemma holds for $n-1$. 
    As in the proof of 
    Proposition \ref{prop:AGs}, the inductive arguments in 
\cite{roehrle:ideal} to derive \cite[Thm.~1.31]{roehrle:ideal} equally apply 
here as well. So that the case for $n-1$ can be obtained as a localization of 
    $\CA_\CI$ at a modular flat of corank $1$ and so the result follows by induction. 
\end{proof}

\section{Rank-generating functions of the poset of regions}
\label{s:rankgenerating}

Let $\CA$ be a 
hyperplane arrangement in the real vector space $V=\BBR^\ell$. 
A \emph{region} of $\CA$ is a connected component of the 
complement $M(\CA) := V \setminus \cup_{H \in \CA}H$ of $\CA$.
Let $\RR := \RR(\CA)$ be the set of regions of $\CA$.
For $R, R' \in \RR$, we let $\CS(R,R')$ denote the 
set of hyperplanes in $\CA$ separating $R$ and $R'$.
Then with respect to a choice of a fixed 
base region $B$ in $\RR$, we can partially order
$\RR$ as follows:
\[
R \le R' \quad \text{ if } \quad \CS(B,R) \subseteq \CS(B,R').
\]
Endowed with this partial order, we call $\RR$ the
\emph{poset of regions of $\CA$ (with respect to $B$)} and denote it by
$P(\CA, B)$. This is a ranked poset of finite rank,
where $\rk(R) := |\CS(B,R)|$, for $R$ a region of $\CA$, 
\cite[Prop.\ 1.1]{edelman:regions}.
The \emph{rank-generating function} of $P(\CA, B)$ is 
defined to be the following polynomial in 
$\BBZ_{\ge 0}[t]$
\begin{equation*}
	\label{eq:rankgen}
	\zeta(P(\CA,B); t) := \sum_{R \in \RR}t^{\rk(R)}. 
\end{equation*}
This poset along with its rank-generating function
was introduced by Edelman 
\cite{edelman:regions}.

Thanks to work of  Bj\"orner, Edelman, and Ziegler 
\cite[Thm.~4.4]{bjoerneredelmanziegler}
(see also Paris \cite{paris:counting}), respectively  
Jambu and Paris \cite[Prop.~3.4, Thm.~6.1]{jambuparis:factored},
in case of a real arrangement $\CA$  
which is supersolvable, 
respectively inductively factored, 
there always exists a suitable base region $B$ so that 
$\zeta(P(\CA,B); t)$
admits a multiplicative decomposition which 
is determined by the exponents of $\CA$, i.e.,
\begin{equation}
	\label{eq:poinprod}
	\zeta(P(\CA,B); t) = \prod_{i=1}^\ell (1 + t + \ldots + t^{e_i}),
\end{equation}
where $\{e_1, \ldots, e_\ell\} = \exp \CA$ is the 
set of exponents of $\CA$.

Quite remarkably
many classical real arrangements do satisfy the
factorization identity \eqref{eq:poinprod}, 
the most prominent ones being Coxeter arrangements.

Let $W = (W,S)$ be a Coxeter group with associated reflection arrangement 
$\CA = \CA(W)$ which consists of the reflecting hyperplanes of 
the reflections in $W$ in the real space $V=\BBR^n$, where $|S| = n$. 
The 
\emph{Poincar\'e polynomial} $W(t)$ of 
the Coxeter group $W$ is the polynomial in $\BBZ[t]$ defined by 
\begin{equation}
	\label{eq:poncarecoxeter}
	W(t) := \sum_{w \in W} t^{\ell(w)},
\end{equation}
where $\ell$ is the length function 
of $W$ with respect to $S$.
Then $W(t)$
coincides with the rank-generating function of the poset of regions $\zeta(P(\CA,B); t)$ of 
the underlying reflection arrangement 
$\CA = \CA(W)$ with respect to $B$ being the dominant Weyl chamber of $W$ in $V$; 
see \cite{bjoerneredelmanziegler} or \cite{jambuparis:factored}.

The following factorization of 
$W(t)$ is due to Solomon \cite{solomon:chevalley}:
\begin{equation}
	\label{eq:solomon}
	W(t) = \prod_{i=1}^n(1 + t + \ldots + t^{e_i}),
\end{equation}
where $\{e_1, \ldots, e_n\}$ is the 
set of exponents of $W$, i.e., the set of exponents of $\CA(W)$. So by the comments above, \eqref{eq:solomon} coincides with the factorization in \eqref{eq:poinprod}.

Moreover, also the rank-generating function of the poset of regions $\zeta(P(\CA_\CI,B); t)$ for 
an ideal arrangement $\CA_\CI$ 
from Definition \ref{def:idealtype} 
also obeys the factorization identity \eqref{eq:poinprod}; see 
\cite{sommerstymoczko}, 
\cite{roehrle:ideal} for partial results, and 
\cite[Cor.~1.3]{abeetal:hess} for the general statement.

Let $W$ be a Coxeter group again with reflection arrangement $\CA = \CA(W)$, 
let $X$ be a member of the intersection lattice $L(\CA)$, and consider the restricted reflection arrangement $\CA^X$. In general, $\CA^X$ is no longer a reflection arrangement. 
Nevertheless, thanks to  \cite[Thm.~1.3]{moellerroehrle:nice},
there always exists a suitable base region $B$ of $\CA^X$ in $X$ so that also 
$\zeta(P(\CA^X,B); t)$ satisfies \eqref{eq:poinprod}, provided $W$ is not of type $E_8$. In case $W$ is of type $E_8$, then $\zeta(P(\CA^X,B); t)$ satisfies \eqref{eq:poinprod}, provided $X$ has rank at most $3$ with only two exceptions or $\CA^X \cong (E_8,D_4)$.

We close with a comment on the
rank-generating function of the poset of regions $\zeta(P(\CA_G,B); t)$ 
of the free  arrangements $\CA_G$ from Theorem \ref{thm:freeAG}.

\begin{remark}
	\label{rem:Hnzeta}
	It follows from Theorems \ref{thm:factoredAG}, \ref{thm:factoredAG2} and \cite[Prop.~3.4, Thm.~6.1]{jambuparis:factored} that 
	$\CA_G$ satisfies \eqref{eq:poinprod} for 
	$G$ a path graph $P_n$ or a path-with-triangle graph $\Delta_{n,1}$ for $n \ge 2$.
		We checked that \eqref{eq:poinprod} also holds for $\CA_G$ for 
	$G = \Delta_{4,2}$ and $\Delta_{5,2}$.
	
	As discussed in the previous section, for $G = A_{n,2} =D_{n+1}$, $A_{5,3} =E_6$, $A_{6,3} =E_7$, and $A_{7,3} =E_8$,  $\CA_G$ is a particular ideal arrangement. Thus it follows from  \cite[Cor.~1.3]{abeetal:hess} that 
		$\CA_G$ satisfies \eqref{eq:poinprod}.
		
	Given this strong evidence, it is very plausible that more generally, $\CA_G$ fulfills \eqref{eq:poinprod} for 
	$G$ an  almost path graph $A_{n,k}$ or a path-with-triangle graph $\Delta_{n,k}$ for any $n \ge 2$ and any $k$.
\end{remark}
 

\bigskip

\addcontentsline{toc}{section}{Acknowledgments}

\noindent {\bf Acknowledgments}: 
The research of this work was supported in part by
the DFG (Grant \#RO 1072/25-1 (project number: 539865068) to G.~R\"ohrle
as well as Grant \#MU 5286/1-1 (project number: 539874788) to P.~Mücksch).
We would like to thank the anonymous referee for several helpful comments improving the exposition of the paper.


\bigskip

\bibliographystyle{amsalpha}

\newcommand{\etalchar}[1]{$^{#1}$}
\providecommand{\bysame}{\leavevmode\hbox to3em{\hrulefill}\thinspace}
\providecommand{\MR}{\relax\ifhmode\unskip\space\fi MR }
\providecommand{\MRhref}[2]{%
  \href{http://www.ams.org/mathscinet-getitem?mr=#1}{#2} }
\providecommand{\href}[2]{#2}


\end{document}